\documentclass[10pt,twocolumn,twoside,journal,a4paper]{IEEEtran}%

\usepackage{amssymb}
\usepackage{amsmath}
\usepackage{amsfonts}
\usepackage{mathdots}

\usepackage[sort]{cite}
\usepackage{psfrag}
\usepackage{epsfig}
\usepackage{amsmath,bbm,times,stmaryrd}

\usepackage{amsthm}

 \usepackage[mathscr]{eucal}

\usepackage{tabularx}
\usepackage{multirow}
\usepackage{enumerate}

\usepackage[normal]{subfigure}

\usepackage{colortbl}
\usepackage{dcolumn}
\newcolumntype{.}{D{.}{.}{1.3}}

\usepackage{txfonts}
\usepackage[subnum]{cases}

\usepackage{equationarray}

\newcommand\dashrule{\leavevmode\xleaders\hbox{-}\hfill\kern0pt}

\def\diag{\operatorname{diag}}

\newcommand{\ba}{{\bvec{a}}}

\newcommand{\bg}{\bvec{g}}
\newcommand{\bh}{\bvec{h}}

\newcommand{\bq}{\bvec{q}}

\newcommand{\bs}{\bvec{s}}

\newcommand{\bu}{\bvec{u}}
\newcommand{\bv}{\bvec{v}}

\newcommand{\bx}{\bvec{x}}

\newcommand{\bA}{{\bf A}}
\newcommand{\bB}{{\bf B}}

\newcommand{\bG}{{\bf G}}
\newcommand{\bH}{{\bf H}}
\newcommand{\bI}{{\bf I}}

\newcommand{\bN}{{\bf N}}

\newcommand{\bP}{{\bf P}}
\newcommand{\bQ}{{\bf Q}}

\newcommand{\bS}{{\bf S}}

\newcommand{\bU}{{\bf U}}

\newcommand{\bX}{{\bf X}}

\newcommand{\calO}{{\mathcal{O}}}

\newcommand{\bbeta}{\mbox{\boldmath $\beta$}}
\newcommand{\bgamma}{\mbox{\boldmath $\gamma$}}

\newcommand{\blambda}{\mbox{\boldmath $\lambda$}}

\newcommand{\bphi}{\mbox{\boldmath $\phi$}}

\newcommand{\bpsi}{\mbox{\boldmath $\psi$}}

\newcommand{\bGamma}{\mbox{\boldmath $\Gamma$}}

\newcommand{\bPhi}{\mbox{\boldmath $\Phi$}}
\newcommand{\bPsi}{\mbox{\boldmath $\Psi$}}

\newcommand{\Real}{\mathbb R}

\newcommand{\1}{\mbox{\boldmath $1$}}

\newcommand{\be}{\begin{eqnarray}}
\newcommand{\ee}{\end{eqnarray}}

\newcommand{\matrixb}{\left[ \begin{array}}
\newcommand{\matrixe}{\end{array} \right]}

\newcommand{\tr}{\mathop{\rm tr}\nolimits}

\def\*{\circledast}

\newtheorem{definition}{Definition}
\newtheorem{remark}{Remark}
\newtheorem{lemma}{Lemma}

\newcommand{\bvec}[1]{\boldsymbol{#1}}

\def\vectorize{\operatorname{vec}}
\newcommand{\vtr}[1]{\vectorize\hspace{-.3ex}\left(#1\right)}

\newcommand{\tensor}[1]{\boldsymbol{\mathscr{\MakeUppercase{#1}}}} 
\newcommand{\tA}{\tensor{A}}
\newcommand{\tB}{\tensor{B}}
\newcommand{\tC}{\tensor{C}}

\newcommand{\tE}{\tensor{E}}

\newcommand{\tG}{\tensor{G}}
\newcommand{\tH}{\tensor{H}}

\newcommand{\tP}{\tensor{P}}

\newcommand{\tX}{\tensor{X}}
\newcommand{\tY}{\tensor{Y}}

\usepackage[vlined,ruled,commentsnumbered]{algorithm2e}

\usepackage{aurical}

\def\bigcircledast{\mathop{\mbox{\fontsize{18}{19}\selectfont $\circledast$}}}
\def\bigCircleDast{\operatornamewithlimits{\bigcircledast}}

\renewcommand{\bigodot}{\mathop{\mbox{\fontsize{18}{19}\selectfont$\odot$}}}

\usepackage{arydshln}

\usepackage{url}

\usepackage{mathtools}

\usepackage{relsize}
\usepackage{amssymb}
\renewcommand{\ltimes}{\mathlarger{\mathlarger{\ltimes}}}

\newcommand{\ttprod}{\bullet}

\usepackage{physics}

\title{Tensor Networks for Latent Variable Analysis: Higher Order Canonical Polyadic Decomposition}
\author{\vspace{-.3ex}Anh-Huy Phan$^{*}$, Andrzej Cichocki,
Ivan Oseledets, Salman Ahmadi Asl, Giuseppe Calvi,
and Danilo Mandic
\thanks{A.-H. Phan, A. Cichocki, I. Oseledets and S. Ahmadi Asl are with Skolkovo Institute of Science and Technology (Skoltech), Russia}
\thanks{G. Calvi and D. Mandic are with Imperial College, London, United Kingdom, email: d.mandic@imperial.ac.uk.}
}

\renewcommand{\CommentSty}[1]{\fontsize{8.7}{9.8}\selectfont\textnormal{\texttt{#1}}\unskip}
\SetKwComment{mtcc}{\% }{}
\SetKwSwitch{Switch}{Case}{Other}{switch}{}{case}{otherwise}{end switch}%

\newcounter{example}

\newenvironment{example}
{\refstepcounter{example}\vspace{10pt}\par\noindent
\textbf{Example \theexample.  }
}
{}%

\makeatletter
\newsavebox{\@brx}
\newcommand{\llangle}[1][]{\savebox{\@brx}{\(\m@th{#1\langle}\)}%
  \mathopen{\copy\@brx\kern-0.5\wd\@brx\usebox{\@brx}}}
\newcommand{\rrangle}[1][]{\savebox{\@brx}{\(\m@th{#1\rangle}\)}%
  \mathclose{\copy\@brx\kern-0.5\wd\@brx\usebox{\@brx}}}
\makeatother

\usepackage{graphicx}
\graphicspath{{./TT_CPD/}}

\begin{document}
\maketitle

\begin{abstract}
The Canonical Polyadic decomposition (CPD) is a convenient and intuitive tool for tensor factorization; however, for higher-order tensors, it often exhibits high computational cost and permutation of tensor entries, these undesirable effects grow exponentially with the tensor order.
Prior compression of tensor in-hand can reduce the computational cost of CPD, but this is only applicable when the rank $R$ of the decomposition does not exceed the tensor dimensions. To resolve these issues, we present a novel method for CPD of higher-order tensors, which rests upon a simple tensor network of representative inter-connected core tensors of orders not higher than 3. 
For rigour, we develop an exact conversion scheme from the core tensors to the factor matrices in CPD, and an iterative algorithm with low complexity to estimate these factor matrices for the inexact case. Comprehensive simulations over a variety of scenarios support the approach.
\end{abstract}

\section{Introduction}

The widespread use of sensor technology and the ever increasing size and complexity  of modern data sets have exposed the limitations of classic linear algebra and the associated flat-view operation of matrix and vector models. This has also highlighted the need for more sophisticated analysis tools capable of coping with the sheer volume associated with Big Data paradigms. Owing to their flexibility and scalability in dealing with multi-way data, higher-order generalizations of matrices, referred to as tensors, have found applications in a wide spectrum of disciplines, ranging from the very theoretical, such as mathematics and physics, to the more practical aspects of signal processing and neuroscience.

The success of tensor algebra has been intimately associated with the efficient way tensor operations deal with the curse of dimensionality. In other words, for tensors in a raw format, the application of standard numerical methods may be intractable, as the required storage memory and a number of operations grow exponentially with the tensor order. To tackle this issue, tensor decompositions aim to represent higher-order tensors through multi-way operations over their latent components. The Canonical Polyadic Decomposition (CPD) is one of such popular methods which factorize a higher-order tensor as the sum of a finite number of rank-one tensors. 
This tensor decomposition was first studied by Hitchcock in 1927\cite{Hitchcock1927}, and was later known as parallel factor analysis (PARAFAC), a tool for chemometric analysis popularized by Harshman \cite{Harshman72}, Carroll and Chang \cite{Carroll_Chang}, and Kruskal\cite{kruskal77}. Since the 1990's, the CPD has attracted attention from signal processing researchers as e.g., the receiving signals in telecommunication and blind source separation often admit the model \cite{Sidiropoulos00Bro,YeredorCaf,Comon20062271,sorensenVDM}. The recent rapid development in machine learning has opened up new applications of CPD in feature extraction, data reconstruction, image completion  \cite{8421043,8305626,BayesCPD-TNNLS,8116755} and various tracking scenarios\cite{sid_adPARAFAC,Phan_tensordeflation_alg}. Compared to other tensor decompositions, the CPD exhibits a great advantage in dimensionality reduction. 
For example, the low-rank tensor approximation of the parameters in convolutional and fully connected layers can accelerate the inference process of convolutional neural networks\cite{DBLP:journals/corr/JaderbergVZ14,DBLP:journals/corr/LebedevGROL14}.
The CPD is also useful for determining the complexity of matrix multiplication, i.e., finding the smallest number of the scalar multiplications required for the multiplication of two matrices; indeed this problem corresponds to finding the rank of certain tensors\cite{Strassen:1969:GEO:2722431.2722798,TICHAVSKY2017362}. 

Over its long history, many researchers have deeply studied the CPD and its properties, including uniqueness and stability.
Efficient algorithms for calculating CPD have also been developed,
and the model has been extended with various additional constraints in order to promote interpretability or to avoid degeneracy. 
%
Despite the great successes, there are still many challenging problems in the CP tensor representation/decomposition.
The main stumbling block in CPD for big data is that its computation for large volume and high order tensors, e.g., those of order $N$ = 10 or 20, is rather complicated and involves numerous technical issues.
First and foremost, most algorithms for CPD rest upon some kind of matricization, that is, the tensor at hand is first flattened to a set of matrices, then the cost function for the tensor decomposition is converted to the objective functions for matrix approximation. While this simplifies the optimization problem and helps to straightforwardly derive update rules for factor matrices, it also gives rise to another issue -- that of permutation of tensor entries. In addition, the higher the order of a tensor, the greater the computational cost for the tensor unfoldings\cite{Phan_FastALS,doi:10.1137/14097968X}.

The curse of dimensionality associated with higher-order data structures also means that computational costs of most existing algorithms for CPD increase exponentially with the tensor order.
For example, the computationally cheapest algorithm for CPD of an order-$N$ tensor which employs the Alternating Least Squares (ALS) updates has a computational cost of $\calO(N R \prod_n I_n)$ or $\calO(N R I^N)$, assuming that the mode dimensions are equal $I_1 = \cdots = I_N = I$ and the estimated rank is $R$\cite{Comon-ALS09,Phan_FastALS}. This explains why most algorithms for CPD are efficient only for tensors of order-3. The CPD for higher-order tensors is therefore routinely performed by reshaping a higher-order tensor into an order-3 tensor followed by a CP decomposition and calculation of the loading components \cite{Phan_FCP,Bhaskara:2014:SAT:2591796.2591881,doi:10.1137/16M1090132}.
Alternatives to a ``direct'' CPD computation have been proposed, such as a prior compression of the tensor using e.g., the Tucker decomposition which can reduce the computational cost of CPD to $\calO(NR^{N+1})$. However, this is only applicable when the rank $R$ of the decomposition is smaller than the tensor dimensions. Moreover, the computational cost still remains high for $N \ge 4$.

This paper presents a novel method for the CP decomposition of higher-order tensors, which is particularly suited to tensors for which tensor rank exceeds tensor dimensions, a prohibitive case for the existing algorithms. The underlying idea behind our approach is to first approximate a tensor by a set of inter-connected core tensors of orders not higher than 3, followed by individual CPDs on such order-3 and low dimension cores. 
To this end, we employ the matrix product states (MPS) \cite{Klumper91,Vidal03}
which is also known as the Tensor Train decomposition \cite{OseledetsTT09}. We demonstrate that the factor matrices within the CPD of the original higher-order tensor are reliably estimated from the compressed TT model. This approach also offers enhanced physical interpretability as demonstrated by a bidirectional mapping between tensors in the CPD and TT formats. For an elegant mathematical formulation, and without a loss of generality, we initially consider the noiseless case, and subsequently extend the approach to accommodate for the presence of noise. For the latter scenario, novel iterative algorithms to estimate the factor matrices within CPD as a result of a prior TT decomposition.

In summary, contributions of this work are as follows:
\begin{itemize}
	\item Compression prior to CPD. This makes our approach possible to perform CPD, even when the rank exceeds tensor dimensions.
	\item For noiseless tensors, an exact mapping from the core tensors of a TT-representation of a given data tensor to the factor matrices of its CPD is established.
	\item For the noisy case, novel iterative algorithms are developed for estimation of factor matrices, with a cost of only $\mathcal{O}(N I R^3)$.
	\item It is demonstrated that the CPD gradients, which are the most computationally expensive operations within CPD algorithms, are now efficiently computed with a computational cost of $\mathcal{O}(I_nR^3)$.
\end{itemize}

The proposed approach is quite general and can serve a wide variety of purposes. For its validation, we consider case studies ranging from basic tensor decompositions to blind identification, blind source separation and low-rank approximations of Hilbert tensors. Comprehensive analysis over rigorous performance metrics conclusively demonstrates the effectiveness of the proposed method in not only efficiently and reliably computing CPDs but also estimating tensor ranks.

 \setlength{\algomargin}{1em}
\begin{algorithm}[t!]
\SetFillComment
\SetSideCommentRight
\CommentSty{\footnotesize}
\caption{The {\tt{TT2CP}} algorithm\label{alg_TT_to_CPD}}
\DontPrintSemicolon \SetFillComment \SetSideCommentRight
\KwIn{Data tensor $\tY$:  $(I_1 \times I_2 \times \cdots \times I_d)$,  and rank-$R$
}
\KwOut{$N$ factor matrices $\bA^{(n)} \in \Real^{I_n \times R}$} \SetKwFunction{mreshape}{reshape}
\SetKwFunction{svds}{svds}
\SetKwFunction{tttocp}{TT\_to\_CPD}
\SetKwFunction{ttcpd}{TT\_CPD}
\Begin{
{\mtcc{{\color{blue}{\bf{Stage 1:}}} TT-decomposition of $\tY$ \dashrule}}
\nl $\tY \approx \tX = \tG_1 \ttprod \tG_2 \ttprod \cdots \ttprod \tG_{N-1} \ttprod \tG_N  $\;
{\mtcc{{\color{blue}{\bf{Stage 2:}}} Convert TT-tensor $\tX$ to a K-tensor of rank-$R$ for the exact model\dashrule}}
\nl $[\bB^{(1)}, \bB^{(2)},\ldots, \bB^{(N)}] = \tttocp(\tX,R)$\;
{\mtcc{{\color{blue}{\bf{Stage 3:}}} Fit a K-tensor of rank-$R$ to TT-tensor \dashrule}}
\nl $[\bA^{(1)}, \bA^{(2)},\ldots, \bA^{(N)}] = \ttcpd(\tX, [\bB^{(1)}, \bB^{(2)},\ldots, \bB^{(N)}])$
}
\end{algorithm}

\section{Preliminaries}
Throughout this paper, the element-wise division, Kronecker, Khatri--Rao (columnwise Kronecker), Hadamard and outer products are denoted, respectively, by $\oslash, \otimes, \odot, \*, \circ$ \cite{cichocki2016tensor}.
A column vector of unities of length $R$ is denoted by $\1_R$.

For convenience, we shall first introduce the definitions of tensor train contraction and tensor train, followed by equivalent representations of a TT-tensor and a Kruskal tensor.

\begin{definition}[Tensor train contraction\cite{Phan_TT_part1}]\label{def_train_contract}\label{def_boxtime}
Consider a tensor $\tA$ of size $I_1 \times I_2 \times \cdots \times I_N$ and a tensor $\tB$ of size  $J_1 \times J_2 \times \cdots \times J_K$.
The tensor train contraction performs a tensor contraction between the last mode of tensor $\tA$ and the first mode of tensor $\tB$, where $I_N = J_1$,
to yield a tensor $\tC = \tA \bullet \tB$ of size $I_1  \times \cdots \times I_{N-1} \times J_2 \times \cdots \times J_K$, the elements of which are given by \\[-1.6em]
\be
c_{i_1,\ldots,i_{N-1},j_2, \ldots,j_K} = \sum_{i_N = 1}^{I_N}  a_{i_1,\ldots,i_{N-1},i_N} \, b_{i_N,j_2, \ldots,j_K}  \notag .
\ee
\end{definition}\vspace{-.8ex}
%
%

\begin{definition}[Kruskal tensor (K-tensor)]\label{def_ktensor}
A Kruskal tensor or K-tensor of order-$N$, denoted by $\tX$, is composed of factor matrices $\bA^{(n)}  = [\ba^{(n)}_1, \ldots , \ba^{(n)}_R] \in \Real^{I_n \times R}$ which have $R$ columns, and is defined as
\be
\tX = \sum_{r= 1}^{R} \lambda_r \, \ba^{(1)}_r \circ \ba^{(2)}_r \circ \cdots \circ \ba^{(N)}_r \notag.
\ee
The K-tensor can also be expressed as
$
\tX = \llbracket \blambda; \bA^{(1)}, \bA^{(2)}, \ldots, \bA^{(N)} \rrbracket $, 
where $\blambda = [\lambda_1, \ldots, \lambda_R]^T >0$ and $\|\ba^{(n)}_r\|^2 = 1$ for  all $n$ and all $r$.
When $\blambda = \1_R$, the K-tensor $\tX$ is simply expressed as
\be
\tX = \llbracket  \bA^{(1)}, \bA^{(2)}, \ldots, \bA^{(N)} \rrbracket \notag.
\ee
\end{definition}

\begin{figure}[t]
\centering
\psfrag{X1}[lb][lb]{\scalebox{1}{\color[rgb]{0,0,0}\setlength{\tabcolsep}{0pt}\hspace{0cm}\begin{tabular}{c}$\tG_1$\end{tabular}}}%
\psfrag{X2}[lb][lb]{\scalebox{1}{\color[rgb]{0,0,0}\setlength{\tabcolsep}{0pt}\hspace{0cm}\begin{tabular}{c}$\tG_2$\end{tabular}}}%
\psfrag{X3}[cb][cb]{\scalebox{1}{\color[rgb]{0,0,0}\setlength{\tabcolsep}{0pt}\hspace{0cm}\begin{tabular}{c}$\tG_3$\end{tabular}}}%
\psfrag{Xn2}[cb][cb]{\scalebox{1}{\color[rgb]{0,0,0}\setlength{\tabcolsep}{0pt}\hspace{0cm}\begin{tabular}{c}$\tG_{N-2}$\end{tabular}}}%
\psfrag{Xn1}[cb][cb]{\scalebox{1}{\color[rgb]{0,0,0}\setlength{\tabcolsep}{0pt}\hspace{0cm}\begin{tabular}{c}$\tG_{N-1}$\end{tabular}}}%
\psfrag{XN}[cb][cb]{\scalebox{1}{\color[rgb]{0,0,0}\setlength{\tabcolsep}{0pt}\hspace{0cm}\begin{tabular}{c}$\tG_{N}$\end{tabular}}}
\psfrag{I1}[cc][cb]{\scalebox{1}{\color[rgb]{0,0,0}\setlength{\tabcolsep}{0pt}\hspace{0cm}\small\begin{tabular}{c}\\[-2em]$I_1$\end{tabular}}}%
\psfrag{I2}[cc][cb]{\scalebox{1}{\color[rgb]{0,0,0}\setlength{\tabcolsep}{0pt}\hspace{0cm}\small\begin{tabular}{c}\\[-2em]$I_2$\end{tabular}}}%
\psfrag{I3}[cc][cb]{\scalebox{1}{\color[rgb]{0,0,0}\setlength{\tabcolsep}{0pt}\hspace{0cm}\small\begin{tabular}{c}\\[-2em]$I_3$\end{tabular}}}%
\psfrag{IN}[cc][cb]{\scalebox{1}{\color[rgb]{0,0,0}\setlength{\tabcolsep}{0pt}\hspace{0cm}\small\begin{tabular}{c}\\[-2em]$I_N$\end{tabular}}}%
\psfrag{In1}[cc][cb]{\scalebox{1}{\color[rgb]{0,0,0}\setlength{\tabcolsep}{0pt}\hspace{0cm}\small\begin{tabular}{c}\\[-2em]$I_{N-1}$\end{tabular}}}%
\psfrag{In2}[cc][cb]{\scalebox{1}{\color[rgb]{0,0,0}\setlength{\tabcolsep}{0pt}\hspace{0cm}\small\begin{tabular}{c}\\[-2em]$I_{N-2}$\end{tabular}}}%
\psfrag{R1}[cc][cc]{\scalebox{1}{\color[rgb]{0,0,0}\setlength{\tabcolsep}{0pt}\hspace{0cm}\footnotesize\begin{tabular}{c}$R_{1}$\end{tabular}}}%
\psfrag{R2}[cc][cc]{\scalebox{1}{\color[rgb]{0,0,0}\setlength{\tabcolsep}{0pt}\hspace{0cm}\footnotesize\begin{tabular}{c}$R_{2}$\end{tabular}}}%
\psfrag{R3}[cc][cc]{\scalebox{1}{\color[rgb]{0,0,0}\setlength{\tabcolsep}{0pt}\hspace{0cm}\footnotesize\begin{tabular}{c}$R_{3}$\end{tabular}}}%
\psfrag{Rn2}[cc][cc]{\scalebox{1}{\color[rgb]{0,0,0}\setlength{\tabcolsep}{0pt}\hspace{0cm}\footnotesize\begin{tabular}{c}$R_{N-2}$\end{tabular}}}%
\psfrag{Rn1}[cc][cc]{\scalebox{1}{\color[rgb]{0,0,0}\setlength{\tabcolsep}{0pt}\hspace{0cm}\footnotesize\begin{tabular}{c}$R_{N-1}$\end{tabular}}}%
%
{\includegraphics[width=.75\linewidth]{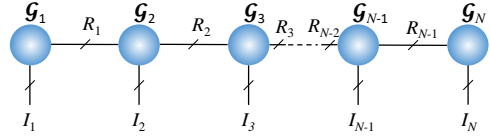}\label{tt_6cores}}
\vspace{-1em}
\caption{Representation of a tensor in the TT-format.}
\label{fig_TT6}
\end{figure}

\begin{definition}[Tensor train or TT-tensor\cite{OseledetsTT09}]\label{def_tt_tensor}
The TT representation of an order-$N$ tensor, $\tX$, of size $I_1 \times I_2 \times \cdots \times I_N$ employs $N$ core tensors, $\tG_1$, $\tG_2$, \ldots, $\tG_N$, whereby every $\tG_n$ is of size $R_{n-1} \times I_n \times R_{n}$, and $R_0 = R_N = 1$, to assume the following form
\be
\tX = \sum_{r_1 = 1}^{R_1} \, \sum_{r_2 = 1}^{R_2}  \cdots \sum_{r_{N-1} = 1}^{R_{N-1}} \tG_1(:,r_1) \circ \tG_2(r_1,:,r_2) \circ \cdots \circ \tG_N(r_{N-1},:), \notag
\ee
where $\tG_n(r_{n-1},:,r_n)$ are vertical fibers of $\tG_n$ and $(R_1, R_2, \ldots,R_{N-1})$ represents the TT-rank of $\tX$ (see also Fig.~\ref{fig_TT6}).
\end{definition}
Since the first and last core tensors, $\tG_1$ and $\tG_N$, are matrices, they can also be represented as $\bG_1$ and $\bG_N$, respectively. 
The TT- decomposition can also be expressed through contractions between the core tensors, $\tG_n$, for $n = 1, \ldots, N$, that is\\[-1.2em]
\be
	\tY \approx  \tG_1 \ttprod \tG_2 \ttprod \cdots \ttprod \tG_{N-1} \ttprod \tG_N   \label{equ_TT_boxcross}.
\ee\\[-1.2em]
Fig.~\ref{fig_TT6} illustrates a representation of an order-$N$ tensor in the TT-format.
A tensor train decomposition can be computed efficiently using the sequential projection algorithm \cite{Vidal03,OseledetsTT09}, or the alternating single or multiple core update algorithm, while the ranks of the decomposition can be determined based on an error bound of the approximation\cite{Phan_TT_part1}.

\begin{figure}[t]
\centering
{\includegraphics[width=1\linewidth,trim = 0.0cm 6cm 0cm 0cm,clip=true]{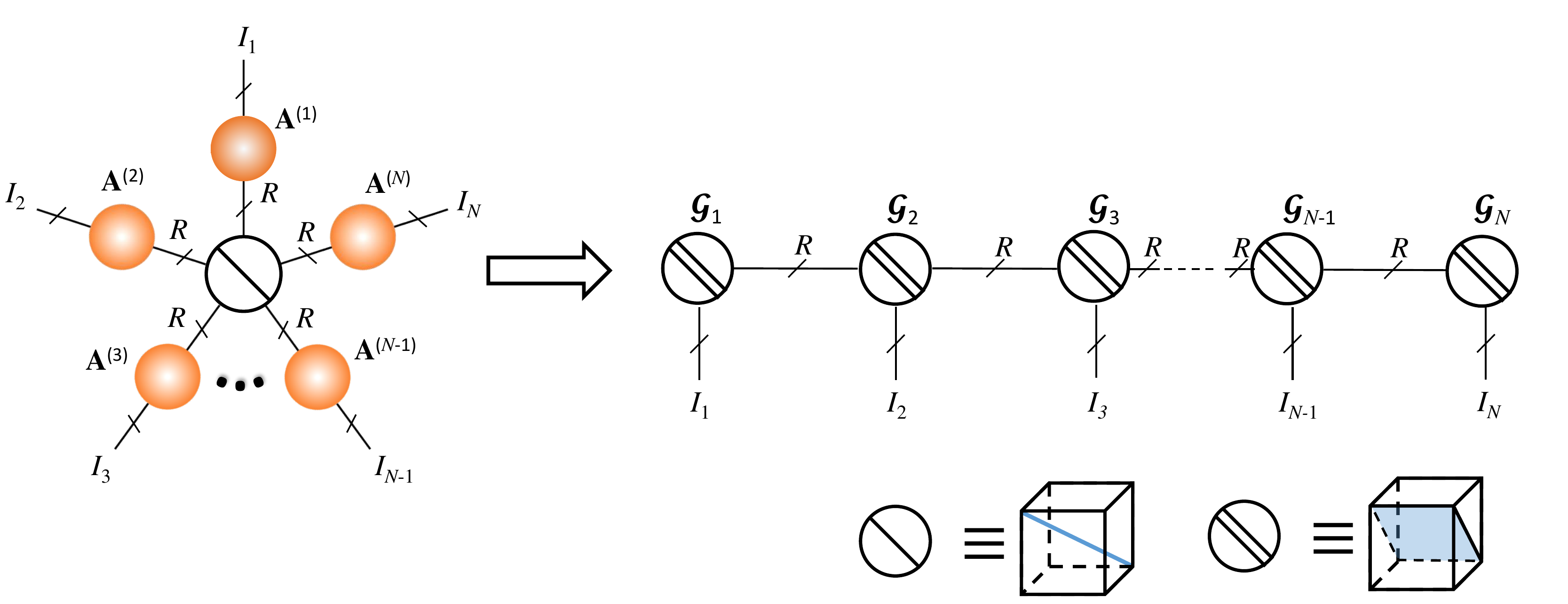}\label{fig_lem1}}
\vspace{-1.5em}
\caption{TT representation of a tensor in the Kruskal format. See also Lemma~\ref{lem_Kt_to_TT}.}
\label{fig_lem1}
\end{figure}

\begin{definition}[Tensor unfolding (flattening, matricization)]\label{def_flat}
The procedure of converting a tensor to a matrix is called tensor unfolding or equivalently flattening or matricization. For an order-$N$ tensor $\tX$, of size $I_1\times I_2 \cdots\times I_N$, its unfolding with respect to mode $n$ is denoted by $\bX_{(n)}$ and yields a matrix of dimension $I_n\times (I_1 I_2 \cdots I_{n-1} I_{n+1} \cdots I_N)$ whose $i_n$-th row represents vectorization of the sub-tensor $\tX(:, \ldots, :,i_n, :,\ldots,:)$.
\end{definition}


\section{Tensor Compression using Tensor Train Decomposition}
\label{sec:TT_compression}

The key idea which underpins our proposed method is to benefit from the super-compression of the large original data tensor through Tensor Train decomposition.

\begin{remark}
The proposed method is physically justified by the properties of the TT-format, whereby any rank-$R$ tensor has an equivalent TT-representation of rank-$(R,\ldots,R)$.
\end{remark}

\subsection{TT-representation of a K-tensor}

\begin{lemma}[\bf TT-representation of a K-tensor \cite{OseledetsTT11}]\label{lem_Kt_to_TT}
A K-tensor $\tY = \llbracket  \bA^{(1)}, \bA^{(2)}, \ldots, \bA^{(N)} \rrbracket $ of rank-$R$ can be expressed in a TT-format in (\ref{equ_TT_boxcross}) as
$\tY =   \tG_1 \ttprod  \tG_2 \ttprod \cdots \ttprod \tG_{N-1} \ttprod \tG_N$,
where the core tensors $\tG_n$ are of size $R \times I_n \times R$, for $n = 2, \ldots, N-1$ and $\tG_1 = \bA^{(1)}$, $\tG_{N} = \bA^{(N)}$.
The $i$-th vertical slices of the core tensors $\tG_n$ are diagonal matrices of the $i$-th rows $\bA^{(n)}(i,:)$ for $i = 1, 2, \ldots, I_n$, that is\\[-1.7em]
\be
	\tG_n(:,i,:) = \diag(\bA^{(n)}(i,:))   \,. \notag
\ee
\end{lemma}
\begin{proof}
For completeness, we provide a brief derivation, although the proof was first provided in \cite{OseledetsTT11}.
Consider a K-tensor \\[-1.8em]
\be
\tY &=&\sum_{r = 1}^{R}  \, \ba^{(1)}_{r}  \circ \ba^{(2)}_{r} \, \circ \cdots \,   \circ  \ba^{(N)}_{r}  \notag\\
&=& \sum_{r_1,r_2, \ldots, r_{N-1}}  \, \ba^{(1)}_{r_1}  \circ \delta_{r_1,r_2} \ba^{(2)}_{r_2} \, \circ \cdots \,   \circ \delta_{r_{N-1},r_{N}}\, \ba^{(N)}_{r_{N}}\, .  \notag
\ee
This implies that the horizontal fibers of the core tensors, $\tG_n$, $n=1,\ldots, N$ are given by
$
\bg^{(n)}_{r_n, r_{n+1}} =    \delta_{r_n,r_{n+1}}   \ba^{(n)}_{r_n}$.
As a result, we have $\tG_n(:,i,:) = \diag(\bA^{(n)}(i,:))$.
\end{proof}
Graphical illustration of the representain in Lemma~\ref{lem_Kt_to_TT} is given in Fig.~\ref{fig_lem1}.
The conversion in Lemma~\ref{lem_Kt_to_TT} indicates that if we fit a TT-tensor $\tG_1 \ttprod  \tG_2 \ttprod \cdots \ttprod \tG_{N-1} \ttprod \tG_N$ of rank-$(R,\ldots,R)$ to a higher-order tensor $\tY$, e.g., using the TT-SVD\cite{Vidal03,OseledetsTT09}, or the Alternating Single/Double/Trible Core Update algorithm\cite{Phan_TT_part1}, then the approximation error of this TT-tensor will not be worse than that of the best rank-$R$ tensor approximation of $\tY$, given by 
\be
	\|\tY -  \tG_1 \ttprod  \tG_2 \ttprod \cdots \ttprod \tG_N\|_F^2 \le \|\tY - \llbracket \bA^{(1)}, \bA^{(2)}, \ldots,  \bA^{(N)} \rrbracket \|_F^2   \notag \,.
\ee
For the exact case, i.e., when $\tY$ is of rank-$R$, we have
\be
	\tG_1 \ttprod  \tG_2 \ttprod \cdots \ttprod \tG_{N-1} \ttprod \tG_N  = \llbracket \bA^{(1)}, \bA^{(2)}, \ldots,  \bA^{(N)} \rrbracket. \notag
\ee

\subsection{Kruskal representation of a TT-tensor}

\begin{figure}[t]
\centering
{\includegraphics[width=1\linewidth,trim = 0.0cm 5cm 0cm 0cm,clip=true]{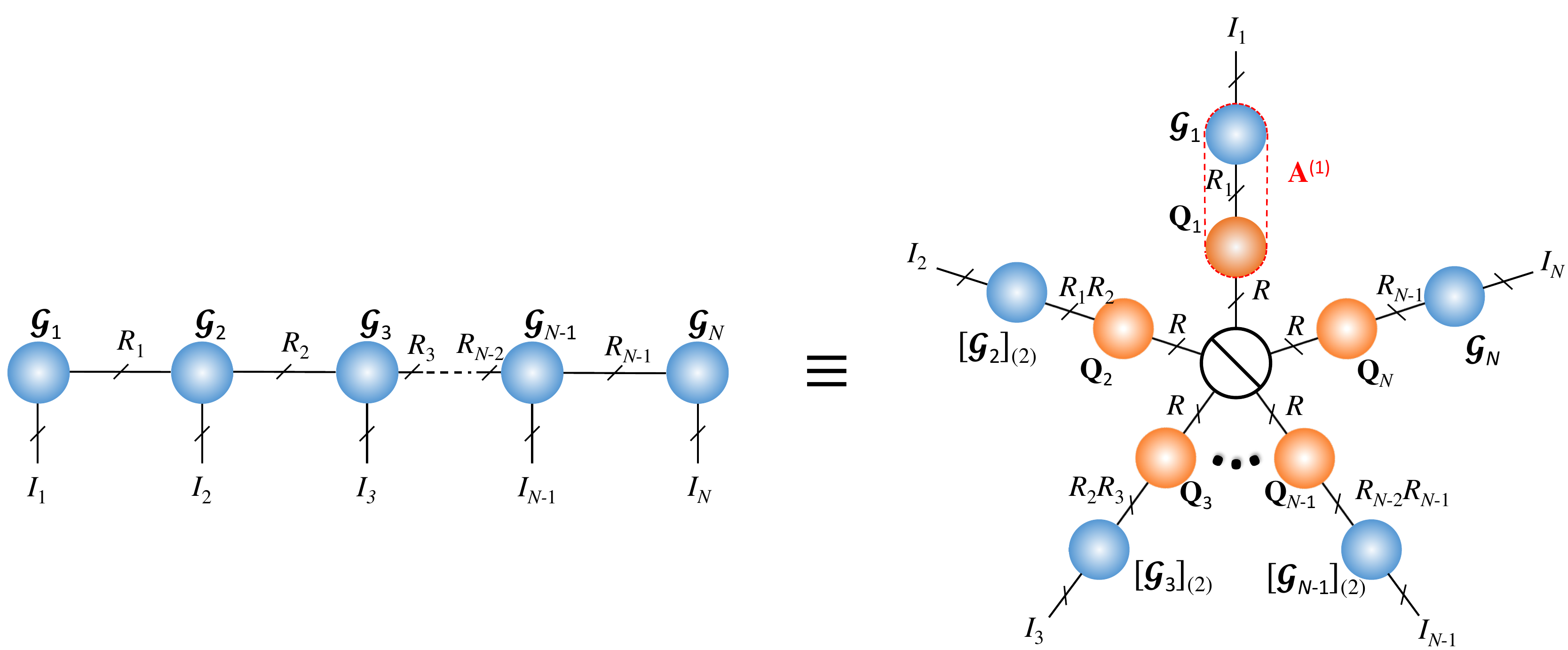}\label{fig_lem2}}
\vspace{-1.5em}
\caption{Kruskal representation of a tensor in the TT-format. $\bQ_n = \1_{R_{>n}}^T \otimes \bI_{R_{n-1}R_{n}} \otimes  \1_{R_{<n-1}}^T$ represents the dependence matrix. See also Lemma~\ref{lem_tt_to_kr}. 
}
\label{fig_lem2}
\end{figure}

\begin{lemma}[\bf Kruskal representation of a TT-tensor]\label{lem_tt_to_kr}
A TT-tensor $\tY =   \tG_1 \ttprod  \tG_2 \ttprod \cdots \ttprod \tG_{N-1} \ttprod \tG_N$ of rank-$(R_0,R_1,R_2, \ldots,R_N)$ has an equivalent Kruskal tensor representation with $R_1 R_2 \cdots R_N$ rank-1 tensors
\be
	 \tY = \llbracket  \bA^{(1)}, \bA^{(2)}, \ldots, \bA^{(N)} \rrbracket \notag
\ee
where the factor matrices $\bA^{(n)}$ are of size $I_n \times R_1 R_2 \cdots R_N$, and are given by
\be
\bA^{(n)} = [\tG_n]_{(2)} \left(\1_{R_{>n}}^T \otimes \bI_{R_{n-1}R_{n}} \otimes  \1_{R_{<n-1}}^T \right) \label{eq_An_tt_to_kr},
\ee
with $R_{<n} = R_0 R_1 \cdots R_{n-1}$ and $R_{>n} = R_{n+1} \cdots R_{N-1} R_N$.
\end{lemma}
The above expression is directly derived from the definition of the TT-tensor, 
and is related to the rank-overlap or CPD with linear dependence \cite{CEM:CEM1206}, where the term $\1_{R_{>n}}^T \otimes \bI_{R_{n-1}R_{n}} \otimes  \1_{R_{<n-1}}^T$ represents the dependence matrix. The equivalence between the TT tensor and its K-tensor is illustrated in Fig.~\ref{fig_lem2}.

\begin{remark}The Kruskal representation of a TT-tensor comprises of $R_1 R_2 \cdots R_N$ rank-1 tensors, which usually exceeds the true rank of the tensor $\tY$.
\end{remark}

\subsection{Towards the exact model: Fast conversion from a TT-tensor to a K-tensor}\label{sub:TT2CP_exactrank}

When tensor is of exact rank-$R$, which is smaller than the tensor dimensions, $I_1, I_2, \ldots, I_N$, then the CPD of such tensor can boil down to Direct Three Linear Decomposition (DTLD) or the Extended DTLD, for which the solution can be found through generalised eigenvalue decomposition \cite{Sanchez1990}. However, this procedure is applicable only when the tensor rank does not exceed tensor dimensions.
In this section, we present a novel method to mitigate this issue and thereby find CPD for higher-order noise-free tensors. More specifically, we propose a direct method to deduce factor matrices of CPD from the core tensors of a TT representation of a rank-$R$ tensor $\tY$.

Note that a TT model does not provide a unique representation,
since e.g., a post-multiplication of the core tensor, $\tG_n$, with any invertible matrix, $\bQ$ of size $R\times R$, and pre-multiplication of the core tensor, $\tG_{n+1}$, with $\bQ^{-1}$ will change the core tensors but preserve the TT-representation of the tensor. In other words,
\begin{align}
&	\tG_1 \ttprod  \cdots \ttprod \tG_n\ttprod \tG_{n+1}  \ttprod   \cdots \ttprod \tG_N
\notag \\
& \quad =	
	\tG_1 \ttprod  \cdots \ttprod \tG_n\ttprod \bQ \ttprod \bQ^{-1} \ttprod \tG_{n+1}  \ttprod   \cdots \ttprod \tG_N  \notag \\
	&\quad  = \tG_1 \ttprod  \cdots \ttprod {\tilde{\tG}}_n\ttprod {\tilde{\tG}}_{n+1}  \ttprod   \cdots \ttprod \tG_N \,  \notag,
\end{align}
where ${\tilde{\tG}}_n =  {\tG}_n  \ttprod \bQ$ and  ${\tilde{\tG}}_{n+1} =  \bQ^{-1} \ttprod \tG_{n+1}$.

Due to this ambiguity, even when the tensor $\tY$ is of exact rank-$R$, fitting a TT model to $\tY$, in general, does not yield a TT-tensor whose core tensors, $\tG_n$, have diagonal structures, as stated in Lemma~\ref{lem_Kt_to_TT}.
In other words, we cannot take the diagonals of $\tG_n(:,i,:)$ as rows of $\bA^{(n)}$.

We next show that $\bA^{(n)}$ are factor matrices of CPDs of $\tG_n$ for $n=2,\ldots,N-1$.

\begin{lemma}\label{lemma_TT_CP_exactmodel}
Assume that a rank-$R$ tensor $\tY$ has a unique CPD given by $ \llbracket \bA^{(1)}, \bA^{(2)}, \ldots,  \bA^{(N)} \rrbracket$, and
a TT representation of rank-$(R,\ldots,R)$, that is
\be
	\tY = \llbracket \bA^{(1)}, \bA^{(2)}, \ldots,  \bA^{(N)} \rrbracket =  \bG_1 \ttprod  \tG_2 \ttprod \cdots\ttprod \tG_{N-1}  \ttprod  \bG_N   . \label{eq_TT_CPD}
\ee
Then $\tG_n$ can be equivalently expressed by K-tensors of $R$ components, the second factor matrices of which are $\bA^{(n)}$, up to scaling and column permutation for $n = 2, \ldots, N-1$.
For example, we have
\be
\tG_n = \llbracket \bQ_{n} , \bA^{(n)}, \bS_n \rrbracket  \notag,
\ee
where $\bQ_{n}$ and $\bS_{n}$ are matrices of size $R \times R$  which hold
\be
\bA^{(1)}  &=& \bG_1 \, \bQ_2  \diag(\bgamma_1) \, \label{eq_lem3_a1},\\
\bA^{(N)} &=& \bG_N^T\, \bS_{N-1}  \diag(\bgamma_{N}) \, , \label{eq_lem3_aN}\\
\bS_n \, \bQ_{n+1} &=& \diag(\bgamma_n), \quad n = 2, \ldots, N-2, \label{eq_lem3_an}
\ee
and $\bgamma_1 \* \bgamma_2 \* \cdots \*\bgamma_{N-2} \*  \bgamma_N = \1_R$.
\end{lemma}

\begin{figure}[t]
\centering
\psfrag{X1}[lb][lb]{\scalebox{1}{\color[rgb]{0,0,0}\setlength{\tabcolsep}{0pt}\hspace{0cm}\begin{tabular}{c}$\tG_1$\end{tabular}}}%
\psfrag{X2}[lb][lb]{\scalebox{1}{\color[rgb]{0,0,0}\setlength{\tabcolsep}{0pt}\hspace{0cm}\begin{tabular}{c}$\tG_2$\end{tabular}}}%
\psfrag{X3}[cb][cb]{\scalebox{1}{\color[rgb]{0,0,0}\setlength{\tabcolsep}{0pt}\hspace{0cm}\begin{tabular}{c}$\tG_3$\end{tabular}}}%
\psfrag{Xn2}[cb][cb]{\scalebox{1}{\color[rgb]{0,0,0}\setlength{\tabcolsep}{0pt}\hspace{0cm}\begin{tabular}{c}$\tG_{N-2}$\end{tabular}}}%
\psfrag{Xn1}[cb][cb]{\scalebox{1}{\color[rgb]{0,0,0}\setlength{\tabcolsep}{0pt}\hspace{0cm}\begin{tabular}{c}$\tG_{N-1}$\end{tabular}}}%
\psfrag{XN}[cb][cb]{\scalebox{1}{\color[rgb]{0,0,0}\setlength{\tabcolsep}{0pt}\hspace{0cm}\begin{tabular}{c}$\tG_{N}$\end{tabular}}}
\psfrag{I1}[cb][cb]{\scalebox{1}{\color[rgb]{0,0,0}\setlength{\tabcolsep}{0pt}\hspace{0cm}\small\begin{tabular}{c}$I_1$\end{tabular}}}%
\psfrag{I2}[cb][cb]{\scalebox{1}{\color[rgb]{0,0,0}\setlength{\tabcolsep}{0pt}\hspace{0cm}\small\begin{tabular}{c}$I_2$\end{tabular}}}%
\psfrag{I3}[cb][cb]{\scalebox{1}{\color[rgb]{0,0,0}\setlength{\tabcolsep}{0pt}\hspace{0cm}\small\begin{tabular}{c}$I_3$\end{tabular}}}%
\psfrag{IN}[cb][cb]{\scalebox{1}{\color[rgb]{0,0,0}\setlength{\tabcolsep}{0pt}\hspace{0cm}\small\begin{tabular}{c}$I_N$\end{tabular}}}%
\psfrag{In1}[cb][cb]{\scalebox{1}{\color[rgb]{0,0,0}\setlength{\tabcolsep}{0pt}\hspace{0cm}\small\begin{tabular}{c}$I_{N-1}$\end{tabular}}}%
\psfrag{In2}[cb][cb]{\scalebox{1}{\color[rgb]{0,0,0}\setlength{\tabcolsep}{0pt}\hspace{0cm}\small\begin{tabular}{c}$I_{N-2}$\end{tabular}}}%
\psfrag{R1}[cc][cc]{\scalebox{1}{\color[rgb]{0,0,0}\setlength{\tabcolsep}{0pt}\hspace{0cm}\footnotesize\begin{tabular}{c}$R_{1}$\end{tabular}}}%
\psfrag{R2}[cc][cc]{\scalebox{1}{\color[rgb]{0,0,0}\setlength{\tabcolsep}{0pt}\hspace{0cm}\footnotesize\begin{tabular}{c}$R_{2}$\end{tabular}}}%
\psfrag{R3}[cc][cc]{\scalebox{1}{\color[rgb]{0,0,0}\setlength{\tabcolsep}{0pt}\hspace{0cm}\footnotesize\begin{tabular}{c}$R_{3}$\end{tabular}}}%
\psfrag{Rn2}[cc][cc]{\scalebox{1}{\color[rgb]{0,0,0}\setlength{\tabcolsep}{0pt}\hspace{0cm}\footnotesize\begin{tabular}{c}$R_{N-2}$\end{tabular}}}%
\psfrag{Rn1}[cc][cc]{\scalebox{1}{\color[rgb]{0,0,0}\setlength{\tabcolsep}{0pt}\hspace{0cm}\footnotesize\begin{tabular}{c}$R_{N-1}$\end{tabular}}}%
%
%
\psfrag{A2}[rb][rb]{\scalebox{1}{\color[rgb]{0,0,0}\setlength{\tabcolsep}{0pt}\hspace{0cm}\begin{tabular}{c}$\bA^{(2)}$\end{tabular}}}%
\psfrag{A3}[rb][rb]{\scalebox{1}{\color[rgb]{0,0,0}\setlength{\tabcolsep}{0pt}\hspace{0cm}\begin{tabular}{c}$\bA^{(3)}$\end{tabular}}}%
\psfrag{An2}[rb][rb]{\scalebox{1}{\color[rgb]{0,0,0}\setlength{\tabcolsep}{0pt}\hspace{0cm}\begin{tabular}{c}$\bA^{(N-2)}$\end{tabular}}}%
\psfrag{An1}[rb][rb]{\scalebox{1}{\color[rgb]{0,0,0}\setlength{\tabcolsep}{0pt}\hspace{0cm}\begin{tabular}{c}$\bA^{(N-1)}$\end{tabular}}}%
%
{\includegraphics[width=1\linewidth,trim = 0.5cm 11.5cm 9cm 0.5cm,clip=true]{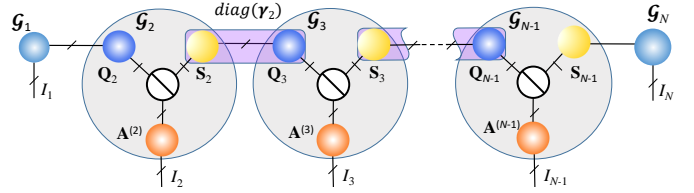}\label{tttocpd}}
\caption{Conversion of a tensor in the TT-format to a K-tensor through CPDs of the 3rd-order core tensors $\tG_2$, \ldots, $\tG_{N-1}$. 
Big nodes designate the core tensors $\tG_n$ and their CPDs, $\tG_n = \llbracket  \bQ_n, \bA^{(n)}, \bS_n \rrbracket$. 
The factor matrix $\bA^{(n)}$ can be retrieved from the 2nd factor matrix. See also Lemma~\ref{lemma_TT_CP_exactmodel}.}
\label{fig_TT6b}
\end{figure}
\begin{proof}
For $n = 2, \ldots, N-1$, we define the following matrices from the unfolding of the train contractions of $\tG_n$
\be
\bG_{<n} &=& \left(\tG_1 \ttprod \tG_2 \ttprod \cdots \ttprod \tG_{n-1} \right)_{(n)}^T  \label{equ_Un},\\
\bG_{>n} &=& \left(\tG_{n+1} \ttprod \tG_{n+2} \ttprod \cdots \ttprod \tG_{N} \right)_{(1)}  \label{equ_Wn},
\ee
and matrices of the Khatri-Rao products
\be
\bA_{<n} &=& \bA^{(n-1)}  \odot \cdots \odot  \bA^{(1)}   \label{equ_Bn},\\
\bA_{>n} &=& \bA^{(N)}  \odot \cdots \odot  \bA^{(n+1)}   \label{equ_Cn}.
\ee
Next we reshape the tensor $\tY$ in (\ref{eq_TT_CPD}) to order-3 tensors $(I_1 \cdots I_{n-1}) \times I_n \times (I_{n+1} \cdots I_N)$ to yield
\be
 \llbracket  \bA_{<n}, \bA^{(n)}, \bA_{>n} \rrbracket = \bG_{<n} \ttprod \tG_n  \ttprod \bG_{>n} 
 \notag.
\ee
On the right-hand side we have a Tucker-2 decomposition, whereby the core tensor $\tG_n$ is multiplied by the matrices $\bG_{<n}$  and $\bG_{>n}^T$ along its respective modes-1 and 3, to yield a rank-$R$ tensor whose second factor matrix is $\bA^{(n)}$. This implies that  $\tG_n$, $n = 2, \ldots, N-1$, can also be expressed by a K-tensor of $R$ components, whose mode-2 factor matrix is $\bA^{(n)}$
\be
\tG_n = \llbracket \bQ_n, \bA^{(n)}, \bS_{n} \rrbracket \,. \notag
\ee
We next show the relation for the factor matrix $\bA^{(1)}$, then derive those for the other factor matrices.
Consider the tensor reshaping of $\tY$ to an order-3 tensor of size $I_1 \times I_2 \times (I_3 \cdots I_N)$ which gives 
\begin{align}
\llbracket \bA^{(1)}, \bA^{(2)}, \bA_{>2} \rrbracket &= \bG_1 \ttprod  \tG_2 \ttprod  \bG_{>2}  \notag \\
&= \bG_1 \ttprod  \llbracket \bQ_2, \bA^{(2)}, \bS_{2} \rrbracket \ttprod \bG_{>2}  \notag
 \, \\\notag
&=  \llbracket \bG_1 \bQ_2, \bA^{(2)}, \bG_{>2} ^T \bS_{2} \rrbracket  \notag
 \, .\notag
\end{align}
The uniqueness of the CP representation of $\tY$ means that $\bA^{(1)}$ and $\bG_1 \bQ_2$ are identical up to scaling of a factor $\bgamma_1$
\be
\bA^{(1)} = \bG_1 \bQ_2 \diag(\bgamma_1)\,.\label{eq_a1_gamma1}
\ee
For the reshaping of $\tY$ which yields a tensor of size $I_1 I_2 \times I_3 \times (I_4 \cdots I_N)$, we have 
\begin{align}
\llbracket \bA^{(2)} \odot \bA^{(1)}, \bA^{(3)}, \bA_{>3} \rrbracket &= \bG_{<3} \ttprod \tG_3 \ttprod  \bG_{>3}  \notag \\
&=  \bG_{<3} \ttprod   \llbracket \bQ_3, \bA^{(3)}, \bS_{3} \rrbracket \ttprod \bG_{>3}  \notag
 \, \\
 &=   \llbracket \bG_{<3} \bQ_3, \bA^{(3)}, \bG_{>3} ^T \bS_{3} \rrbracket \label{eq_unfold_213}.
\end{align}
Since $\tG_1 \ttprod \tG_2 = \llbracket \bG_1 \bQ_2, \bA^{(2)},   \bS_{2} \rrbracket$, its mode-3 unfolding  is given by
\be
\bG_{<3} = (\bG_1 \ttprod \tG_{2})_{(3)}^T =  (\bA^{(2)} \odot \bG_1 \bQ_2) \, \bS_{2}^T. \notag 
\ee
Upon inserting into the CP representation in (\ref{eq_unfold_213}), we obtain 
\begin{align}
\llbracket \bA^{(2)} \odot \bA^{(1)}, \bA^{(3)}, \bA_{>3} \rrbracket &=  \llbracket (\bA^{(2)} \odot \bG_1 \bQ_2) \, \bS_{2}^T \bQ_3, \bA^{(3)}, \bG_{>3} ^T \bS_{3} \rrbracket \notag .
\end{align}
Again, due to the uniqueness of the CPD on the left-hand side, and from (\ref{eq_a1_gamma1}), the factor matrices in both CPDs are identical up to the scalling of a factor $\bbeta$, that is
\be
\bA^{(2)} \odot \bA^{(1)} &=& (\bA^{(2)} \odot \bG_1 \bQ_2) \, \bS_{2}^T \bQ_3 \diag(\bbeta)\notag \\
&=& (\bA^{(2)} \odot \bA^{(1)} \diag(\bgamma_1)^{-1}) \, \bS_{2}^T \bQ_3 \diag(\bbeta)\notag \\
&=& (\bA^{(2)} \odot \bA^{(1)}) \,  \diag(\bgamma_1)^{-1} \bS_{2}^T \bQ_3 \diag(\bbeta)\notag.
\ee
The last expression implies that $\bS_{2}^T \bQ_3$ must be a diagonal matrix, $
\bS_{2}^T \bQ_3 = \diag(\bgamma_2)$.
Similarly, we can prove that $\bS_{n}^T \bQ_{n+1} = \diag(\bgamma_n)$, for $n = 3, \ldots, N-2$, are diagonal matrices.
By replacing the K-tensor representation of $\tG_n$ into (\ref{eq_TT_CPD}), we arrive at an alternative Kruskal representation of $\tY$, given by
\begin{align}
\tY &= \bG_1 \ttprod  \tG_2 \ttprod  \tG_3  \ttprod \cdots \ttprod  \tG_{N-1} \ttprod  \bG_N
\notag \\
&= \bG_1 \ttprod  \llbracket \bQ_2, \bA^{(2)}, \bS_{2} \rrbracket \ttprod  \llbracket \bQ_3, \bA^{(3)}, \bS_{3} \rrbracket  \ttprod
\cdots \ttprod \notag \\&\quad\quad\quad\quad \quad\quad\quad\quad\quad\quad\quad \ttprod \llbracket \bQ_{N-1}, \bA^{(N-1)}, \bS_{N-1} \rrbracket
 \ttprod  \bG_N \notag\\
&= \llbracket \bG_1 \bQ_2, \bA^{(2)} \,  \bS_{2}^T \bQ_3 ,  \bA^{(3)} \, \bS_{3}^T \bQ_{4},  \ldots,\notag \\
& \quad\quad\quad\quad \quad\quad\quad\quad\quad\quad\quad   \bA^{(N-2)} \bS_{N-2}^T \bQ_{N-1}, \bA^{(N-1)} ,  \bG_N^T \bS_{N-1}  \rrbracket \, .\notag
\end{align}
Since the CPD of $\tY$ is assumed to be unique, the above K-tensor of $\tY$ must be identical to $\llbracket \bA^{(1)}, \bA^{(2)}, \ldots,  \bA^{(N)} \rrbracket$ up to the scaling of the factor matrices. In other words, we obtain the representation given in (\ref{eq_lem3_a1})-(\ref{eq_lem3_an}). This completes the proof.
\end{proof}

\begin{remark}
As a consequence of Lemma~\ref{lemma_TT_CP_exactmodel}, we can immediately deduce the factor matrices, $\bA^{(n)}$, from only $(N-2)$ low-scale CPDs of the 3rd-order core tensors, $\tG_2$, \ldots, $\tG_{N-1}$, that is
\be
\tG_{n} = \llbracket \blambda_n; \bQ_n , \bA^{(n)}, \bS_n \rrbracket  \notag ,
\ee
where $\blambda_n$ are positive scaling vectors of length $R$, while the columns of the factor matrices are of unit length.
Since $\tG_n$ are order-3 tensors, their CPDs can be found in closed-form through DTLD.
\end{remark}

Another important issue is the permutation ambiguity, which refers to a possible arbitrary ordering of the columns of the factor matrices $\bA^{(n)}$ in CPDs of $\tG_n$, which may not match the ordering of columns of the other factor matrices.
This requires us to reorder the columns of $\bA^{(1)}$, $\bA^{(2)}$, \ldots, $\bA^{(N)}$ using appropriate permutations.
Fortunately, these permutations can be determined through the products $\bS_n^T \bQ_{n+1}$ for $n = 2, \ldots, N-2$. 
In practice, due to the scaling ambiguity of the decompositions, these products are not always diagonal matrices as stated in Lemma~\ref{lemma_TT_CP_exactmodel}, but exhibit some form of permutation, that is
\be
\bS_n^T \bQ_{n+1} =  \diag(\bgamma_n) \, \bP_n. \notag
\ee
By identifying dominant entries in the rows of $\bS_n^T \bQ_{n+1}$ and their locations, we can determine $\bgamma_n$ and the permutation matrices $\bP_n$.
As a result, the loading components of the K-tensor $\tG_{n+1} = \llbracket \blambda_{n+1}; \bQ_{n+1} , \bA^{(n+1)}, \bS_{n+1} \rrbracket$ can be permuted and normalised to give
\be
\tG_{n+1} =
 \llbracket \blambda_{n+1} \* \bP_n \, \bgamma_n; \bQ_{n+1} \bP_n^T \diag(\1_R \oslash \bgamma_n) , \bA^{(n+1)} \, \bP_n, \bS_{n+1} \, \bP_n \rrbracket  \notag
\ee
so that
\be
\bS_n^T \, {\tilde{\bQ}_{n+1}} = \bS_n^T \, \bQ_{n+1} \bP_n^T  \diag(\1_R \oslash \bgamma_n)  =  \bI_R \,. \notag
\ee
This correction can be achieved sequentially for $\bA^{(2)}$, $\bA^{(3)}$, \ldots, $\bA^{(N-1)}$, while the first and the last factor matrices, $\bA^{(1)}$ and $\bA^{(N)}$, are simply obtained as matrix products
\be
\bA^{(1)} =  \bG_1 \bQ_2  \, ,  \quad
\bA^{(N)}  = \bG_{N}^T \bS_{N-1}  \, . \notag
\ee
Finally, the scaling vectors $\blambda_2$, $\blambda_2$, \ldots, $\blambda_{N-1}$ can be absorbed into one of the  factor matrices, e.g., $\bA^{(N)}$.

After permuting and re-scaling $\bA^{(n)}$, we finally obtain a rank-$R$ CPD of the high order tensor $\tY$ through $(N-2)$ CPDs of order-3 tensors.
The entire procedure to construct a K-tensor of rank-$R$ from a TT-tensor for the exact model is summarized in Algorithm~\ref{alg_TT_to_CPD_exact}, with the corresponding tensor graph given in Fig.~\ref{fig_TT6b}.

{
 \setlength{\algomargin}{1em}
\begin{algorithm}[t!]
\SetFillComment
\SetSideCommentRight
\CommentSty{\footnotesize}
\caption{{TT to K-tensor conversion for the exact model}\label{alg_TT_to_CPD_exact}}
\DontPrintSemicolon \SetFillComment \SetSideCommentRight
\KwIn{TT-tensor $\tG_1 \ttprod \tG_2 \ttprod \cdots \ttprod \tG_N$:  $(I_1 \times I_2 \times \cdots \times I_N)$ of rank-$(R,\ldots,R)$}
\KwOut{A K-tensor $\llbracket  \bA^{(1)}, \bA^{(2)}, \ldots, \bA^{(N)}  \rrbracket$ of rank $R$}
\Begin{
{\mtcc{{\color{blue}{Rank-$R$ CPD of $\tG_n$ to find $\bA^{(2)}$, \ldots, $\bA^{(N-1)}$}}\dashrule}}
\For {$n = 2, \ldots, N-1$}{
\nl $\tG_n \approx  \llbracket    \blambda_n ; \bQ_n, \bA^{(n)}, \bS_{n}  \rrbracket$\;
}

\For {$n = 2, \ldots, N-2$}{
{\mtcc{{\color{blue}{Seek permutation matrices $\bP_n$}}\dashrule}}
\nl $\bS_{n}^T \, \bQ_{n+1} \approx \diag(\bgamma_n) \bP_n $\;
{\mtcc{{\color{blue}{Reorder columns of $\bA^{(n+1)}$}} \dashrule}}
\nl $\bA^{(n+1)} \leftarrow \bA^{(n+1)} \bP_n$, $\bS^{(n+1)} \leftarrow \bS^{(n+1)} \bP_n$, $\blambda_n \leftarrow \blambda_n \* \bP_n \bgamma_n$
}
\nl $\blambda =  \blambda_2 \* \blambda_3  \* \cdots \* \blambda_{N-1}$\;
\nl $\bA^{(1)} = \bG_1 \, \bQ_2$, $\bA^{(N)} = \bG_N^T \, \bS_{N-1} \diag(\blambda)$
}
\end{algorithm}
}

Although Algorithm~\ref{alg_TT_to_CPD_exact} is derived for the exact (noise-free) model, it can also be applied to the noisy cases, and even used as an efficient initialization method for CPD of higher-order tensors.
Example~\ref{ex_n5and10_I5_R5_10} illustrates the efficiency of Algorithm~\ref{alg_TT_to_CPD_exact} for CPD of noisy tensors.

\subsection{Sequential conversion based on best rank-1 matrix approximation}\label{sec:seqconv}

We next show that the factor matrices, $\bA^{(n)}$ in Lemma~\ref{lemma_TT_CP_exactmodel}, can be found through best rank-1 approximations to the slices of the core tensors $\tG_n$.
Different from the previous section, the rank-$R$ tensor $\tY = \llbracket \bA^{(1)}, \bA^{(2)}, \ldots,  \bA^{(N)} \rrbracket$ can be represented by a TT-tensor of $(N-2)$ core tensors of order-3 as follows
\be
\tY = \tG_1 \ttprod  \tG_2 \ttprod \cdots \ttprod \tG_{N-2} \notag ,
\ee
where $\tG_1$ is of size $I_1 \times I_2 \times R$,
$\tG_n$ is of size $R \times I_{n+1}  \times R$, $n = 2, \ldots, N-3$
and  $\tG_{N-2}$ is of size $R \times I_{N-1} \times I_{N}$.
Similar to the result in Lemma~\ref{lemma_TT_CP_exactmodel}, the core tensors $\tG_n$ have representations in the Kruskal format
\be
\tG_1 &=&  \llbracket \bA^{(1)}, \bA^{(2)}, \bS_1\rrbracket \,,\notag \\
\tG_n &=&  \llbracket \bQ_n, \bA^{(n+1)}, \bS_n\rrbracket  \,, \quad n = 2, \ldots, N-3\notag \,,\\
\tG_{N-2} &=&  \llbracket \bQ_{N-2}, \bA^{(N-1)}, \bA^{(N)} \rrbracket   \,, \notag
\ee
where $\bS_n^T \, \bQ_{n+1} = \diag(\bgamma_n)$, for $n = 1, 2, \ldots, N-3$.
Observing that a multiplication of  the core tensor $\tG_n$ with the matrix $\bS_{n-1}$ yields a K-tensor whose first loading matrix
is an identity matrix, we can write
\be
\tilde{\tG}_n   &=& \bS_{n-1}^T \ttprod \tG_n
=     \llbracket \bS_{n-1}^T \bQ_{n}, \bA^{(n+1)}, \bS_n\rrbracket \notag  \\
&=&   \llbracket \diag(\bgamma_{n-1}), \bA^{(n+1)}, \bS_n\rrbracket \notag  \\
&=&   \llbracket \bI_R , \bA^{(n+1)}, \tilde{\bS}_n   \rrbracket \notag   ,
\ee
where $\tilde{\bS}_n    = \bS_n \diag(\bgamma_{n-1})$. 

\begin{remark}
The above particular decomposition implies that the horizontal slices, $\tilde{\tG}_n(r,:,:)$, of $\tilde{\tG}_n$ are rank-1 matrices and the columns of the factor matrix $\bA^{(n+1)}$ are leading singular vectors of these slices, that is
\be
\tilde{\tG}_n(r,:,:) = \ba^{(n+1)}_r  \, \tilde{\bs}^{(n) T}_r \, . \label{eq_bestrank1}
\ee
\end{remark}

We shall next utilise the above property to propose an alternative method to derive the factor matrices, $\bA^{(n)}$, as follows.
First, we decompose the core tensor $\tG_1$ to find the two factor matrices, $\bA^{(1)}$ and $\bA^{(2)}$, 
then modify the second core tensor $\tG_2$ by the third factor matrix of $\tG_1$, i.e., $\bS_1$, to give $\tilde{\tG}_2  = \bS_{1}^T \ttprod \tG_2$.
The third factor matrix, $\bA^{(3)}$, and the matrix ${\bS}_2$ are found through the best rank-1 matrix approximation to the horizontal slices of the tensor $\tilde{\tG}_2$ in (\ref{eq_bestrank1}).
We continue the process by modifying the next core tensors $\tG_3$, \ldots, $\tG_{N-2}$, and computing the best rank-1 matrix approximations.
The last factor matrix $\bA^{(N)}$ is the matrix ${\bS}_{N-2}$. The whole estimation procedure is outlined in Algorithm~\ref{alg_TT_to_CPD_exact_svd}.
Different from Algorithm~\ref{alg_TT_to_CPD_exact}, this algorithm runs only one CPD of the first core tensor $\tG_1$, and need not deal with the permutation of the factor matrices $\bA^{(n)}$. It is worth noting that both proposed conversion algorithms are also applicable to complex-valued tensors.

{
 \setlength{\algomargin}{1em}
\begin{algorithm}[t!]
\SetFillComment
\SetSideCommentRight
\CommentSty{\footnotesize}
\caption{{Sequential construction of a K-tensor from a TT- representation}\label{alg_TT_to_CPD_exact_svd}}
\DontPrintSemicolon \SetFillComment \SetSideCommentRight
\KwIn{TT-tensor $\tY = \tG_1 \ttprod \tG_2 \ttprod \cdots \ttprod \tG_{N-2}$:  $(I_1 \times I_2 \times \cdots \times I_N)$ of rank-$(R,\ldots,R)$}
\KwOut{A K-tensor $\tY = \llbracket  \bA^{(1)}, \bA^{(2)}, \ldots, \bA^{(N)}  \rrbracket$ of rank $R$}
\Begin{
\nl $\tG_1 \approx  \llbracket  \bA^{(1)}, \bA^{(2)}, \bS_{1}  \rrbracket$\tcc*{find $\bA^{(1)}$  and  $\bA^{(2)}$}
\For {$n = 2, \ldots, N-2$}{
\nl $\tilde{\tG}_n = \bS_{n-1}^T \,  \ttprod \tG_n$\tcc*{Modify $\tG_n$}
{\mtcc{{\color{blue}{Seek the best rank-1 approximations \dashrule}}}}
\nl  \lFor {$r = 1, \ldots, R$}{
  $\tilde{\tG}_n(r,:,:) = \ba^{(n+1)}_r \, \bs_{r}^{(n) T}$
}
}
\nl  $\bA^{(N)} =  \bS_{N-2}$
}
\end{algorithm}
}

\section{An Iterative Algorithm to Fit a Rank-$R$ tensor to a TT-tensor}\label{sec::TT2CP_inexact}

We now derive an iterative algorithm which fits a K-tensor to a TT-tensor. This algorithm is used after a TT-compression of the data tensor, as in Stage 3 in Algorithm~\ref{alg_TT_to_CPD}.

Given a TT-tensor $\tG =  \tG_1 \ttprod  \tG_2  \ttprod \cdots \ttprod \tG_N$ where the core tensors can be either complex-valued or real-valued tensors,
the following cost function is minimised to find its best rank-$R$ tensor $\tA =  \llbracket \bA^{(1)}, \bA^{(2)}, \ldots,  \bA^{(N)} \rrbracket$, that is
\be
\min \quad D = \frac{1}{2} \|   \tG_1 \ttprod  \tG_2 \ttprod \cdots \ttprod \tG_N  - \llbracket \bA^{(1)}, \bA^{(2)}, \ldots,  \bA^{(N)} \rrbracket  \|_F^2 \label{equ_cost_fitTT} \, .
\ee
An obvious approach would be to replace the TT-tensor $\tG = \tG_1 \ttprod  \tG_2 \ttprod \cdots \ttprod \tG_N$ in (\ref{equ_cost_fitTT}) by an equivalent Kruskal tensor for which the factor matrices are found acording to Lemma~\ref{lem_tt_to_kr}
in the form
\be
\bU^{(n)} = [\tG_n]_{(2)} \left(\1_{R_{>(n)}}^T \otimes \bI_{R_{n-1}R_{n}} \otimes  \1_{R_{<n-1}}^T  \right) \,.\label{eq_Un_tt_kt}
\ee
Then, most existing algorithms for CPD can be applied to the problem of low-rank approximation of the Kruskal tensor $\llbracket \bU^{(1)}, \bU^{(2)}, \ldots,  \bU^{(N)} \rrbracket$ in order to minimise 
\be
\min \quad D = \frac{1}{2} \|\llbracket \bU^{(1)}, \bU^{(2)}, \ldots,  \bU^{(N)} \rrbracket  - \llbracket \bA^{(1)}, \bA^{(2)}, \ldots,  \bA^{(N)} \rrbracket  \|_F^2 \notag \,.
\ee
The trick here is to exploit the gradients
\be
\frac{\partial D}{\partial \bA^{(n)}} &=& \bU^{(n)} \left(\bigodot_{k\neq n} \bU^{(k)}\right) \left(\bigodot_{k\neq n} \bA^{(k)}\right) \notag \\
&=&  \bU^{(n)} \left(\bigCircleDast_{k\neq n} \bU^{(k) T}\bA^{(k)} \right)\,\label{eq_gradient_An}.
\ee
For example, the ALS update rule for CPD is expressed as
\be
\bA^{(n)} = \bU^{(n)} \left(\bigCircleDast_{k\neq n} \bU^{(k) T}\bA^{(k)} \right)  \left(\bigCircleDast_{k\neq n} \bA^{(k) T}\bA^{(k)} \right)^{-1} \notag.
\ee
We note that the Tensor toolbox in \cite{Bader_kolda} implements this computational trick, while algorithms for CPD in this toolbox and the TENSORBOX \cite{tensorbox} support the decomposition of Kruskal tensors. However,
the existing algorithms do not exploit linear dependence of the structured Kruskal tensor $\llbracket \bU^{(1)}, \bU^{(2)}, \ldots,  \bU^{(N)} \rrbracket$, while the factor matrices $\bU^{(n)}$ in principle have a relatively high number of columns $R_1 R_2 \cdots R_N$.
Such algorithms are therefore not optimized for the decomposition of structured Kruskal tensors.
We next derive algorithms for the optimization problem in (\ref{equ_cost_fitTT}).  A fast computation method which fully exploits the linear dependence structure in (\ref{eq_Un_tt_kt}) is presented in Appendix~\ref{sec:fastCP_structured_tensor},

\subsection{The ALS algorithm}\label{sec:ALS}

For generality, we consider complex-valued tensors.
To derive an ALS algorithm which sequentially updates $\bA^{(n)}$ while fixing the other factor matrices, we express the inner product between a TT-tensor and a K-tensor as
\begin{align}
{\langle \tG, \tA \rangle}
&=    \langle \bG_{<n} \ttprod \tG_n \ttprod \bG_{>n},  \llbracket \bA_{<n}, \bA^{(n)}, \bA_{>n}  \rrbracket    \rangle  \notag \\
&=     \langle  \tG_n,  \llbracket \bG_{<n}^H \bA_{<n}, \bA^{(n)}, \bG_{>n}^{*} \bA_{>n}  \rrbracket  \rangle \notag \\
&=   \tr\left( \left[{\tG_n}  \right]_{(2)}  \left(\bG_{>n} \bA_{>n}^{*}  \odot   \bG_{<n}^T \bA_{<n}^{*} \right)    \bA^{(n) \, H}   \right)^{*}  \notag \\
&=  \tr\left( \left[{\tG_n}  \right]_{(2)}  \left(\bPsi_{>n}  \odot   \bPsi_{<n} \right)    {\bA^{(n)}}^H  \right)^{*}, \label{equ_cost_An}
\end{align}
where the symbols ``$H$'' and ``*'' denote respectively the Hermitian conjugate and complex conjugate, while $\bG_{<n}$, $\bG_{>n}$, $\bA_{<n}$ and $\bA_{>n}$ are defined in (\ref{equ_Un})-(\ref{equ_Cn}),  and
\be
\bPsi_{>n} &=&  \bG_{>n} \bA_{>n}^{*}, \\
\bPsi_{<n} &=& \bG_{<n}^T \bA_{<n}^{*}, \\
\bGamma_n &=&  \bigcircledast_{k \neq n} ({\bA^{(k)}}^H \bA^{(k)}) \,.
\ee
The cost function can now be rewritten as
\begin{align}
D &= \frac{1}{2} \left( \|\tG \|_F^2  + \|\tA\|_F^2  - 2 \Re\{{\langle \tG, \tA \rangle} \} \right) \notag \\
&=  \frac{1}{2} \left( \|\tG \|_F^2  + \tr({\bA^{(n)}}  \bGamma_n^T  {\bA^{(n)}}^H)
\right. \notag \\ & \quad \quad\quad\quad  \left. - 2   \Re\left\{\tr\left( \left[{\tG_n}  \right]_{(2)}  \left(\bPsi_{>n}  \odot   \bPsi_{<n} \right)    {\bA^{(n)}}^H  \right)\right\} \right)  \label{equ_cost_An},
\end{align}
for $n = 1, 2, \ldots, N$.
Since the cost function is quadratic in $\bA^{(n)}$, its solution is given explicitly by
\be
{\bA^{(n)}} = \left[{\tG_n}  \right]_{(2)}   \,  \left(\bPsi_{>n}   \odot   \bPsi_{<n} \right)  ({\bGamma_n^{*}})^{-1} \, . \label{equ_update_An}
\ee

The above update rules also holds for $\bA^{(1)}$ and $\bA^{(N)}$, for which $\bPsi_{<1} = \bPsi_{>N} = \1_R^T$.
The factor matrices $\bA^{(n)}$ are updated one by one sequentially. Each time, we need to compute the two contraction matrices, $\bPsi_{>n}$ and $\bPsi_{<n}$, of size $R\times R$, and invert a symmetric matrix, $\bGamma_n$, of size $R\times R$.
We note that the matrix $\bPsi_{>n}$ represents contraction between the core tensors $\tG_{n+1}$, \ldots, $\tG_{N}$, i.e.,  on the right side of $\tG_n$, with the factor matrices $\bA^{(n+1)}$, \ldots, $\bA^{(N)}$, which are on the right side of $\bA^{(n)}$.
Similarly, $\bPsi_{<n}$ is contraction of the core tensors $\tG_k$ and the factor matrices $\bA^{(k)}$, where $k = 1, \ldots, n-1$, i.e, on the left side of $n$. Although the update rule in (\ref{equ_update_An}) is relatively simple, the computation of the left and right contraction matrices $\bPsi_{<n}$ and $\bPsi_{>n}$ when running $n$ from 1 to $N$ is expensive. For example, when the tensor dimensions are identical, i.e., $I_1 = \cdots = I_N = I$, the computational cost of $\bG_{<n}$ is $\calO((I^2 + I^3 + \cdots + I^{n-1})R^2)$, whereas $\bA_{<n}$ computes the Khatri-Rao product of $(n-1)$ matrices; hence, it has a computational cost of $\calO(I^{n-1} R)$. This indicates that $\bPsi_{<n}$ requires a computational cost of at least $\calO(I^{n-1} R^2)$. The right contraction matrices $\bPsi_{>n}$ also require a cost of $\calO(I^{N-n+1} R^2)$.
Bearing in mind that the computational cost of each update in the ALS algorithm for CPD is of order $\calO(I^N R)$, this means that there is not much reduction in computational cost of the update rule in (\ref{equ_update_An}) compared with the oridinary ALS update.

The next section introduces a method to significantly reduce computational cost of the update rule in (\ref{equ_update_An}).

\subsection{Progressive computation of the contraction matrices $\bPsi_{>n}$ and $\bPsi_{<n}$}\label{sec::computeRnLn}

The most computationally expensive aspect of the update rule in (\ref{equ_update_An}) is the computation of  the left and right contraction matrices $\bPsi_{<n}$ and  $\bPsi_{>n}$. Fortunately, these matrices can be efficiently computed through a successive computation with a cost of $\calO(I_nR^3)$.

 \begin{lemma}[Progressive computation $\bPsi_{>n}$ and $\bPsi_{<n}$\label{lem_progressive_RnLn}]
 The contraction matrices, $\bPsi_{>n}$ and $\bPsi_{<n}$, can be computed with a cost of $\calO(I_nR^3)$ using the following recursive formula
 \be
\bPsi_{>n}  &=&   \left[\tG_{n+1}   \right]_{(1)}    \left( \bPsi_{>(n+1)} \odot  \bA^{(n+1) *} \right)    \, \label{equ_Wn_update_},  \\
\bPsi_{<n}  &=&   \left[  \tG_{n-1} \right]_{(3)}  \left(\bA^{(n-1) *}  \odot  \bPsi_{<(n-1)} \right) \label{equ_update_Ln_}.
 \ee
 \end{lemma}
 \begin{proof}
From the definition of $\bG_{>n}$  in (\ref{equ_Wn}), we can rewrite this matrix as
 \be
 \bG_{>n} &=&  \left[\tG_{n+1} \ttprod \tG_{n+2}  \ttprod \cdots  \ttprod \tG_{N} \right]_{(1)} \notag \\
 &=&  \left[\tG_{n+1}   \ttprod \bG_{>(n+1)}   \right]_{(1)} \notag\\
 &=&  \left[\tG_{n+1}   \right]_{(1)}   \left(\bG_{>(n+1)} \otimes \bI_{I_{n+1}} \right) \,. \notag
 \ee
 Similarly, from (\ref{equ_Cn}),  $\bA_{>n}$ can be rewritten as a Khatri-Rao product of $\bA_{>(n+1)}$ and $\bA^{(n+1)}$ in the form
 \be
 \bA_{>n} =  \bA_{>(n+1)}  \odot \bA^{(n+1)} \,.\notag
 \ee
By replacing the above expressions for $\bG_{>n}$ and $\bA_{>n}$  into $\bPsi_{>n} = \bG_{>n} \bA_{>n}^{*}$, we obtain a recursive formula
to efficiently compute $\bPsi_{>n}$ as
\be
\bPsi_{>n} &=& \bG_{>n} \bA_{>n}^{*} =  \left[\tG_{n+1}   \right]_{(1)}   \left(\bG_{>(n+1)} \otimes \bI_{I_{n+1}} \right) \left(\bA_{>(n+1)} ^{*} \odot \bA^{(n+1) * } \right) \notag \\
&=&  \left[\tG_{n+1}   \right]_{(1)}    \left( \bG_{>(n+1)}  \bA_{>(n+1)}^{*} \odot  \bA^{(n+1) * } \right) \notag \\
&=&  \left[\tG_{n+1}   \right]_{(1)}    \left( \bPsi_{>(n+1)} \odot  \bA^{(n+1) *} \right)    \, . \notag 
\ee
Similarly, $\bPsi_{<n}$ can be expressed as
%
\be
\bPsi_{<n} &=& \bG_{<n}^T \bA_{<n}^{*}  = \left[ \bG_{<(n-1)} \ttprod \tG_{n-1} \right]_{(3)} \, \left(\bA^{(n-1) * }  \odot \bA_{<(n-1)}^{*}\right)   \notag \\
&=& \left[  \tG_{n-1} \right]_{(3)}  \left(\bI_{I_{n-1}} \otimes \bG_{<(n-1)} \right)^T  \left(\bA^{(n-1) *}  \odot \bA_{<(n-1)}^{*} \right)  \notag \\
&=&  \left[  \tG_{n-1} \right]_{(3)}  \left(\bA^{(n-1) *}  \odot  \bPsi_{<(n-1)} \right) \,.\notag 
\ee
It is now straightforward to see that the computation of $\bPsi_{>n}$ from $\bPsi_{>(n+1)}$, or $\bPsi_{<n}$ from $\bPsi_{<(n-1)}$ comes at a computational cost of $\calO(I_nR^3)$.
\end{proof}

\subsection{Update strategy and the entire algorithm}\label{sec::fastALS}

As above, in order to update $\bA^{(n)}$ using the update rule in (\ref{equ_update_An}), we need to compute the two contraction matrices $\bPsi_{>n}$ and $\bPsi_{<n}$. Although these matrices can be computed at a cost of $\calO(I_nR^3)$, they are updated from different sides. The right-contraction matrices are updated right-to-left, i.e., $\bPsi_{>n}$ is computed from $\bPsi_{>(n+1)}$, whereas the left-contraction matrices are updated left-to-right, i.e., $\bPsi_{<n}$ is computed from $\bPsi_{<(n-1)}$. Therefore, when updating $\bA^{(1)}$, $\bA^{(2)}$, \ldots, $\bA^{(N)}$ from left to right, sequentially, we can update the left-contraction matrices $\bPsi_{<n}$, but may need to fully compute the right contraction matrices $\bPsi_{>n}$. Similarly, when updating the factor matrices from right-to-left sequentially, i.e, $\bA^{(N)}$, \ldots, $\bA^{(2)}$, $\bA^{(1)}$, we may need to fully compute $\bPsi_{<n}$, but can update $\bPsi_{>n}$.

In order to fully exploit the progressive computation of $\bPsi_{<n}$ and $\bPsi_{>n}$, and thus further reduce the computational cost of the update rule (\ref{equ_update_An}), we employ the following two-side update strategy
\be
\bA^{(1)}, \bA^{(2)}, \ldots, \bA^{(N-1)}, \bA^{(N)}, \bA^{(N-1)}, \ldots, \bA^{(2)}, \bA^{(1)}, \bA^{(2)}, \ldots \notag
\ee
The estimation procedure which implements the above update order is described in Algorithm~\ref{alg_TT_to_CPD_als}.
A computational trick here is that the right-contraction matrices $\bPsi_{>(N-1)}$, \ldots, $\bPsi_{>2}$, $\bPsi_{>1}$ are precomputed at line 1, before the iterative process starts. For convenience, we denote $\bPsi_{<1} = \bPsi_{>N} = \1_R^T$ row vectors of ones.
The algorithm will first update $\bA^{(n)}$ from left-to-right with $n$ running from 1 to $(N-1)$, then sequentially update $\bA^{(N)}$, $\bA^{(N-1)}$, \ldots, $\bA^{(2)}$, and so on.

After updating $\bA^{(1)}$, we update the next left-contraction matrix $\bPsi_{<2}$. For updating $\bA^{(2)}$, we need not compute $\bPsi_{<2}$ and $\bPsi_{>2}$, as these are available from the previous update and precomputation. However, we will update $\bPsi_{<3}$ after obtaining a new estimate $\bA^{(2)}$.
The left-to-right update procedure is applied similarly to the other factor matrices, and is switched to the right-to-left update after updating $\bA^{(N-1)}$.

When executing the right-to-left estimation process, $n = N, N-1, \ldots, 2$, we update only the right-contraction matrices $\bPsi_{>n}$ from the previous one, i.e., $\bPsi_{>(n+1)}$, while the left-contraction matrices $\bPsi_{<n}$ are available from the left-to-right update procedure.

\begin{remark}
Together with the progressive computation of $\bPsi_{<n}$ and $\bPsi_{>n}$ and the two-sides update strategy, updating $\bA^{(n)}$ requires to update either $\bPsi_{<n}$ or $\bPsi_{>n}$ with a cost of $\calO(I_n R^3)$. The product $\left[{\tG_n}  \right]_{(2)}   \,  \left(\bPsi_{>n}   \odot   \bPsi_{<n} \right)$ comes at a computational cost $\calO(2I_nR^3)$. Therefore, computational cost of each iteration to update $\bA^{(n)}$ is of order $\calO(I_n R^3)$, and is much lower than that of the ordinary ALS $\calO(R \prod_{n} I_n )$ for higher-order tensors.
\end{remark}

{
 \setlength{\algomargin}{1em}
\begin{algorithm}[t!]
\SetFillComment
\SetSideCommentRight
\CommentSty{\footnotesize}
\caption{{An iterative algorithm to fit a rank-$R$ tensor to a TT-tensor}\label{alg_TT_to_CPD_als}}
\DontPrintSemicolon \SetFillComment \SetSideCommentRight
\KwIn{TT-tensor $\tG_1 \ttprod \tG_2 \ttprod \cdots \ttprod \tG_N$:  $(I_1 \times I_2 \times \cdots \times I_N)$ of rank-$(R,\ldots,R)$}
\KwOut{A K-tensor $\llbracket \bA^{(1)}, \bA^{(2)}, \ldots, \bA^{(N)}  \rrbracket$ of rank $R$}
\Begin{
{\mtcc{{\color{blue}{Precompute contraction matrices $\bPsi_{>n}$,  $\bPsi_{>N} = \1_R^T$}}\dashrule}}
\nl \lFor {$n = N-1, \ldots, 1$}{
 $\bPsi_{>n}  =  \left[\tG_{n+1}   \right]_{(1)}    \left( \bPsi_{>(n+1)} \odot  \bA^{(n+1) *} \right)$
}
\Repeat{a stopping criterion is met}{
{\mtcc{{\color{blue}{Update $\bA^{(n)}$ from left to right}}\dashrule}}
\For {$n = 1, \ldots, N-1$}{
\nl ${\bA^{(n)}} = \left[{\tG_n}  \right]_{(2)}   \,  \left(\bPsi_{>n}   \odot   \bPsi_{<n} \right)  ({\bGamma_n^{*}})^{-1} $\;
\nl $\bPsi_{<(n+1)}  =  \left[  \tG_{n} \right]_{(3)}  \left(\bA^{(n)  *}  \odot  \bPsi_{<n} \right)$
}
{\mtcc{{\color{blue}{Update $\bA^{(n)}$ from right to left }}\dashrule}}
\For {$n = N, \ldots, 2$}{
\nl ${\bA^{(n)}} = \left[{\tG_n}  \right]_{(2)}   \,  \left(\bPsi_{>n}   \odot   \bPsi_{<n} \right)  ({\bGamma_n^{*}})^{-1} $\;
\nl $ \bPsi_{>(n-1)} = \left[\tG_{n}   \right]_{(1)}    \left( \bPsi_{>n} \odot  \bA^{(n)  *} \right)   $ 
}
}
}
\end{algorithm}
}

Finally, to complete the algorithm, we introduce an efficient expression for the computation of the cost function in (\ref{equ_cost_fitTT}) or (\ref{equ_cost_An}) in the form
\be
D = \frac{1}{2} \left(\|\tG\|_F^2 - \tr\left(\bA^{(n) H} \, \left[{\tG_n}  \right]_{(2)}   \,  \left(\bPsi_{>n}   \odot   \bPsi_{<n} \right) \right) \right)  \, ,   \notag \label{fast_assess_cost}
\ee
for $n = 1, 2, \ldots, N$, and with the complexity of $\calO(I_nR^2)$ .

\section{Extensions of Other Optimization Algorithms for CPD}

In addition to the ALS algorithm, we can derive other iterative algorithms for the optimization problem in (\ref{equ_cost_fitTT}), e.g., the Levenberg-Marquardt (LM) algorithm\cite{Phan_fLM}.
The results in Lemma~\ref{lem_progressive_RnLn} are useful for fast computation of the gradients of the objective function with respect to the factor matrices, i.e., the product of the mode-$n$ matricization of the TT-tensor $\tG$ and the Khatri-Rao product of all-but-one factor matrices $\bA^{(k)}$, $k \neq n$. Following (\ref{equ_update_An}), these terms are given by
\be
 [\tG]_{(n)} \left(\bigodot_{k \neq n} \bA^{(n)} \right) = \left[{\tG_n}  \right]_{(2)}   \,  \left(\bPsi_{>n}   \odot   \bPsi_{<n} \right) \, \label{eq_gradients}.
 \ee
For the LM update rule, we can exploit a fast inversion of the Hessian matrix in \cite{Phan_fLM} with a cost of $\calO(R^6)$. Implementation of the LM algorithm for the decomposition in (\ref{equ_cost_fitTT}) is similar to that of the fLM algorithm for CPD \cite{Phan_fLM}, except for the gradients computed in (\ref{eq_gradients}).

Another advantage of the proposed method is that it quickly provides a good estimate for the constrained CPD, e.g., the error preserving correction method or the CPD with a bounded error \cite{Phan_EPC} in the form
\be
\min  \;\; \sum_{r} \|\ba^{(1)}_{r}  \circ \ba^{(2)}_{r} \, \circ \cdots \,   \circ  \ba^{(N)}_{r} \|_F^2  \;\; \text{s.t.} \;\;\quad \|\tY - \tX\|_F^2  \le  \varepsilon^2 \notag
\ee
or the CPD with bounded norm of rank-1 tensors
\be
\min \;\; \quad \|\tY - \tX\|_F^2  \;\;
\text{s.t.} \;\; \sum_{r} \|\ba^{(1)}_{r}  \circ \ba^{(2)}_{r} \, \circ \cdots \,   \circ  \ba^{(N)}_{r} \|_F^2 \le  \delta^2 \notag \,.
\ee

\section{Simulations and Results}\label{sec:simulation}

Our proposed framework was evaluated over case studies spanning a variety of technical problems to verify its advantages in estimating tensor ranks, $R$. For rigour, this was achieved even when $R$ exceeds the largest tensor dimension, that is $R>\text{max}(I_n)$, $n=1, \dots, N$, where $N$ is the order of the tensor. As an illustrative example, the case of a simple decomposition of random tensors is first considered, followed by more practical scenarios including blind source identification in communication systems, and blind source separation. Finally, the advantages of the proposed method in tensor compression are demonstrated through low-rank approximation of a Hilbert tensor.

\begin{example}{\bf Decomposition of random tensors.}{\label{ex_n5and10_I5_R5_10}}
\begin{figure}[t]
\centering
\psfrag{XRankConv}[lc][lc]{\scalebox{1}{\color[rgb]{0,0,0}\setlength{\tabcolsep}{0pt}\hspace{0cm}\footnotesize\begin{tabular}{c}Alg.~\ref{alg_TT_to_CPD_exact}\end{tabular}}}%
\subfigure[$R = 5$]{\includegraphics[width=.49\textwidth,trim = 0.0cm .0cm 0cm 0cm,clip=true]{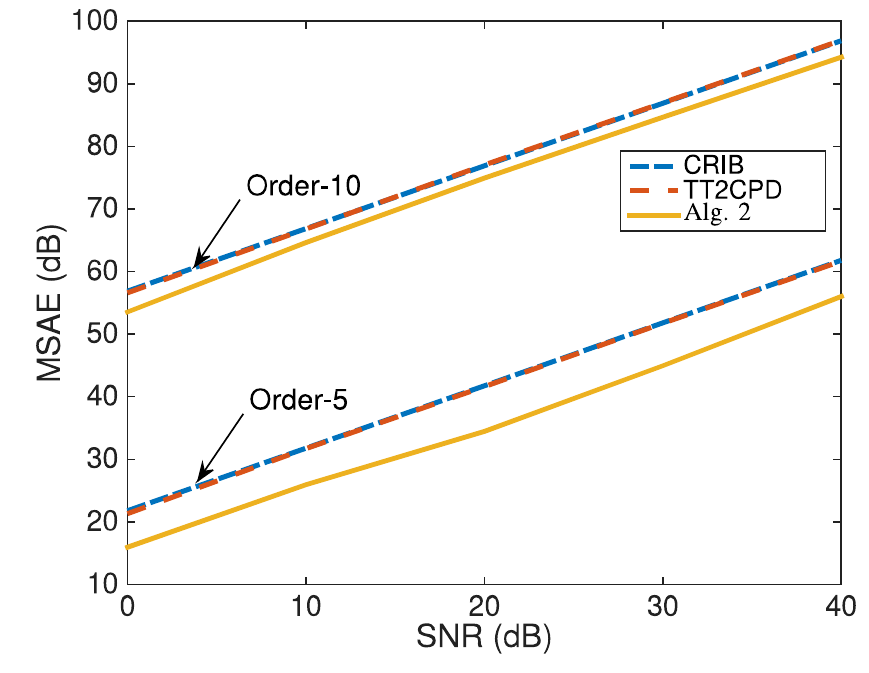}}
\\
\subfigure[$R = 10$]{\includegraphics[width=.49\textwidth,trim = 0.0cm .0cm 0cm 0.0cm,clip=true]{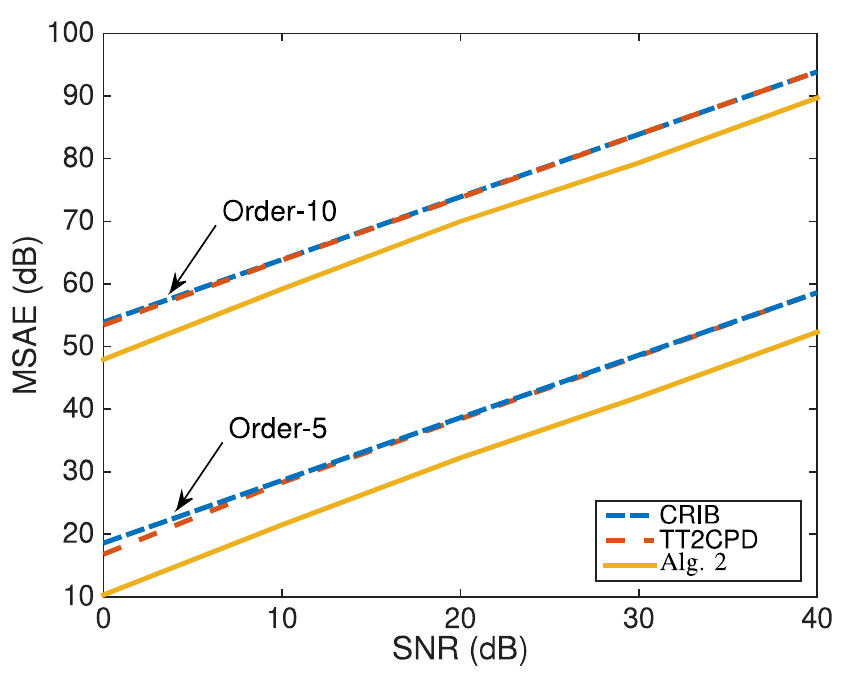}}
\caption{The MSAE of components for CPD of order-5 and order-10 tensors of size $5 \times 5 \times \cdots \times 5$ in Example~\ref{ex_n5and10_I5_R5_10}.}
\label{fig_n5_10_i5_r10}
\end{figure}
The effectiveness of Algorithms 2 and 3 was validated through a decomposition of random noisy tensors, to serve as a representative example and provide a physical intuition behind the approach.

We considered order-$N$ noisy tensors
\be
\tY =  \llbracket \bA^{(1)}, \bA^{(2)}, \ldots, \bA^{(N)}\rrbracket + \tE \,  \notag
\ee
which are of size $5 \times 5 \times \cdots \times 5$ and rank-$R = 5$ or $R = 10$, where $N = 5$ or 10. Additive Gaussian noise tensor, $\tE$, was added to $\tY$ to yield the noise levels SNR = 0, 10, 20, 30 or 40 dB.
Tensor $\tY$ was  approximated by TT-tensors, whose the highest TT-rank is $R$; then rank-$R$ K-tensors were constructed using Algorithm~\ref{alg_TT_to_CPD_exact}.
The performance was assessed through the Mean Squared Angular Error (MSAE)\footnote{$SAE(\bx,\hat{\bx}) = -20\log_{10}\arccos \frac{\bx^T {\hat{\bx}}}{\|\bx\|_2 \|\hat{\bx}\|_2}  \quad (dB)$} in (dB) and for various test cases: the tensor rank $R = 5$ and 10, the tensor order $N = 5$ and 10, nearly noise-free case, SNR = 40 dB, and heavy noise case, SNR = 0 dB.
The MSAEs were compared with the Cram\'er-Rao induced bounds (CRIB) \cite{Petr_CRIB} in Fig.~\ref{fig_n5_10_i5_r10} and indicate that even in the presence of noise, the MSAEs of estimated components by Algorithm~\ref{alg_TT_to_CPD_exact} were only a few dB lower than theoretical CRIBs.
For example, at SNR = 40 dB, the MSAEs were on average 5.6 dB 
lower than the CRIB for the decomposition of order-5 tensors of rank-5, and 2.7 dB for the decomposition of order-10 tensors  of the same rank.
With the tensor order of $N=5$ and the rank $R = 10$, i.e.,  when the rank exceeded tensor dimensions $I_n = 5$, the differences of SAE were slightly higher, 6.3 dB for order-5 tensors, and 4.1 dB for order-10 tensors.
The MSAEs in (dB)  of the components were found to linearly decrease with the SNRs (increase in noise power).
The estimation accuracy of the factor matrices was significantly improved and attained the CRIB when using Algorithm~\ref{alg_TT_to_CPD}.
%
%
 \end{example}

\begin{example}{\bf Blind identification (BI) in a system of $2$ mixtures and $R$ binary signals.}\label{ex_qam_bi_2xR}
\end{example}

The problem of blind source identification is of great relevance in wireless communications. As the signals transmitted by different users correspond to rank-1 terms in the case of line-of-sight propagation \cite{Sidiropoulos00Bro}, the use of CPD becomes natural.
We considered a linear system which consists of $I$ sensors and receives signals $\bX=\bH \bS +\bN$ from $R$ stationary sources, $\bS$, in the presence of additive noise, $\bN$ (see Fig.~\ref{fig_bi_columnstyle} for a general principle).
Given only the knowledge of the noisy observations, the task is to estimate the mixing matrix, $\bH \in \Real^{I \times R}$,  under some mild assumptions, i.e., the sources are statistically independent and non-Gaussian, their number is known, and the matrix $\bH$ has no pairwise collinear columns (see also \cite{YeredorCaf,Comon20062271}).%

We applied the well-known approach proposed in \cite{YeredorCaf,Comon20062271} which creates a higher-order tensor, $\tY$, generated from the observations, $\bX$, by means of partial derivatives of the second Generalised Characteristic Functions (GCFs) of the observations, $\Phi_{\bx}(\bu) = \log{\mbox{E} \left[\exp(\bu^T \bx) \right]}$, at multiple processing points, $\bu$ of length $I$
\be
\Psi_{\bx}(\bu) &=& \frac{\partial^N \Phi_{\bx}(\bu) }{\partial \bu^N}  =  \frac{\partial^N \Phi_{\bs}(\bH^T \bu)}{\partial \bu^N}
 \notag \\
&=&  \Psi_{\bs}(\bH^T\bu) \times_1 \bH  \times_2 \bH \cdots \times_N \bH \,, \notag \label{eq_CP_gcf}
\ee
where $\Psi_{\bs}(\bv)$ are the $N$th-order derivatives of $\Phi_{\bs}(\bv)$ with respect to a vector, $\bv$, of the length $R$, which results in an $N$th-order diagonal tensor, because the sources are statistically independent.
More detail on the generation of the high order derivative tensors is presented in \cite{cichocki2017tensor}.

 \begin{figure}[t!]
\centering
\includegraphics[width=.95\linewidth, trim = 0.0cm .0cm 0cm 0cm,clip=true]{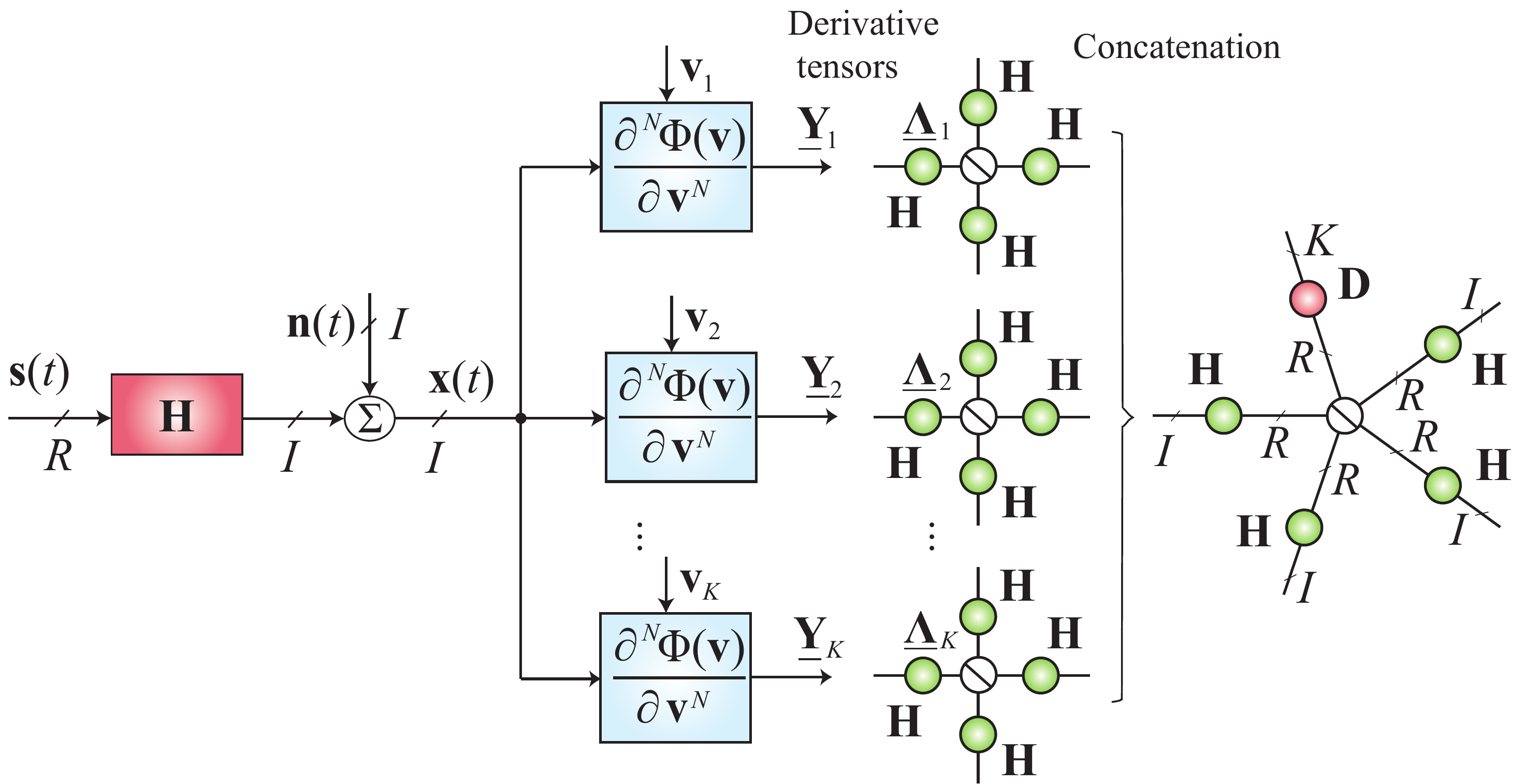}\label{fig_bi_columnstyle1}
\caption{Tensorization and tensor-based approach to blind identification. The task is to estimate the mixing system $\bH$ from only the knowledge of the noisy observations $\bX$.
A high dimensional tensor $\tY$ is generated from the observations $\bX$ by means of higher-order statistics (cumulants) or partial derivatives of the second generalised characteristic functions of the observations. A decomposition of $\tY$ by CP, INDSCAL, CONFAC or BC decomposition allows us to retrieve the mixing matrix $\bH$.}
\label{fig_bi_columnstyle}
\end{figure}

The mixing system in this example consisted of two mixtures, $I = 2$, linearly composed by $R$ signals of length $T = 100  \times \, 2^R$, the entries of which took the values 1 or -1, i.e., $s_{r,t}  =  1$ or $-1$. The mixing matrix $\bH$ of size $2 \times R$ was randomly generated, where $R = 3, 4, \ldots, 9$. The Gaussian noise was added to the mixtures $\bX = \bH \, \bS$, to yield the signal-to-noise ratio SNR = 20 dB.

We constructed 52 derivative tensors $\tP_i^{(N)} = \Psi^{(N)}_{\bx}(10 \bu_i)$ for each derivative order-$N = 3, \ldots, 8$, evaluated at 52 processing points, $\bu_i$, where $i = 1, 2\, \ldots, 52$.
The first two processing points $\bu_1$ and $\bu_2$ were two leading left singular vectors of $\bX$, while $\bu_3, \ldots, \bu_{52}$ were generated such that $\bu_i =  c_i \, \bu_1  +  \sqrt{1-c_i^2} \, \bu_2$, where $c_i = \bu_1^T \bu_i$ were uniformly distributed over a range of $[-0.99, 0.99]$, i.e., $c_i= -0.99, -0.9496, -0.9092, \ldots, 0.99$.

Next, from the derivative tensors $\tP^{(N)}_{i}$, we constructed 50 tensors, $\tY_i$, $i = 1, 2, \ldots, 50$, of $(N+1)$th-order and of size $R\times \cdots \times R \times 3$ as follows
\be
	\tY_i(:,\ldots,:,1) &=& \tP_1^{(N)}  - \bar{\tP}^{(N)}, \quad \notag \\
	\tY_i(:,\ldots,:,2) &=& \tP_2^{(N)}	- \bar{\tP}^{(N)}, \quad  \notag \\
	\tY_i(:,\ldots,:,3) &=& \tP_{i+2}^{(N)} - \bar{\tP}^{(N)} \notag,
\ee
where $\bar{\tP}^{(N)} = \frac{1}{50}\sum_{i = 3}^{52} \tP_{i}^{(N)}$.

To estimate the mixing matrix from 50 CPDs of $\tY_i$ of rank-$R$, we applied the CPD with prior TT-compression.
For each estimation, we computed the mean of Squared Angular Errors (MSAE) $SAE(\bh_r, \hat{\bh}_r) = -20\log_{10}\arccos(\frac{\bh_r^T \hat{\bh}_r}{|\bh_r|_2 |\hat{\bh}_r|_2})$ over all columns $\bh_r$.
The mean over 50 MSAEs indicated the average accuracy of estimations of $\bH$, while the maximum of 50 MSAEs indicated the best estimation performance with a suitably chosen processing point, $\bu_i$, combined with $\bu_1$ and $\bu_2$.

Fig.~\ref{fig_bi_2xR_tt2cp} illustrates the performance over 100 runs for $R  = 3, 4, \ldots, 9$. With a suitably chosen processing point, $\bu_i$, the decomposition of the derivative tensors yielded good estimation of the mixing matrix. Moreover, performances with a prior TT-decomposition were more stable and yielded an approximately 2 dB higher MSAE than those using the standard CPD for the derivative tensors of orders 7 and 8 and for a high number of sources.


{
\begin{figure}[t!]
\centering
\subfigure[CPD]{
\includegraphics[width=.85\linewidth, trim = 0.0cm 0cm 0cm .6cm,clip=true]{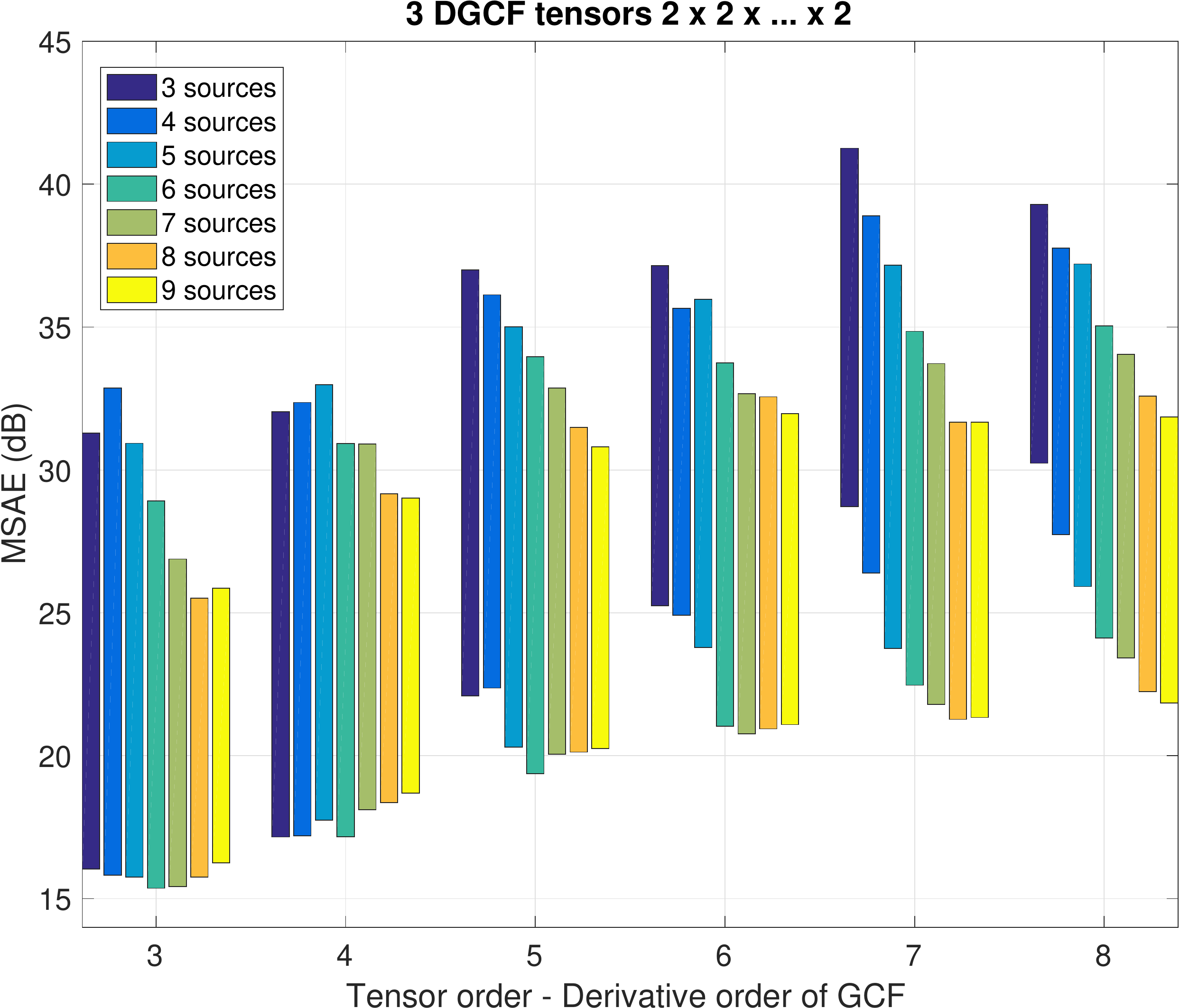}}
\subfigure[TT2CDP]{
\includegraphics[width=.85\linewidth, trim = 0.0cm  0cm 0cm .6cm,clip=true]
{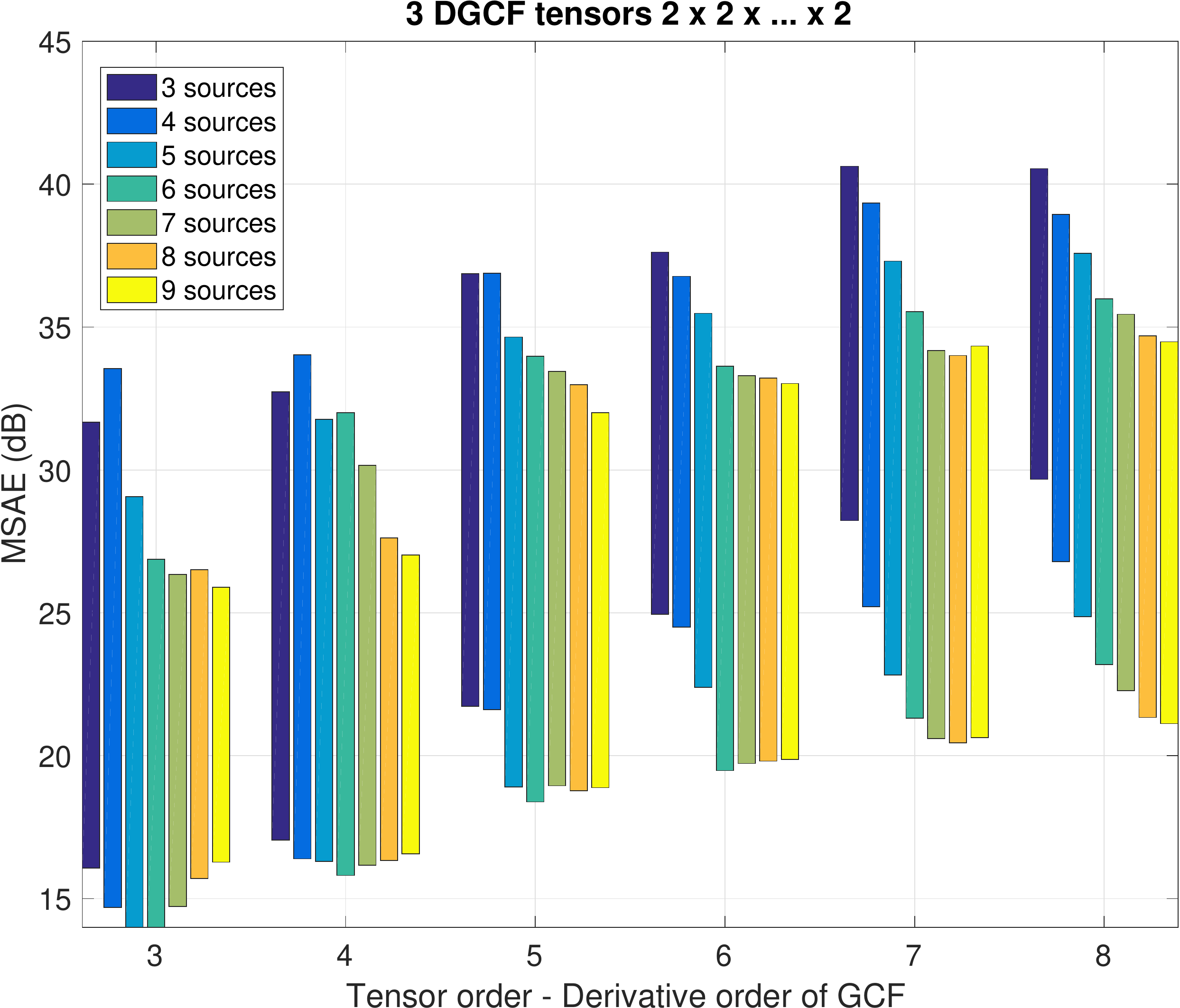}}
\caption{Performance of CPD and the proposed TT2CPD approach
(a) Mean SAE (in dB) in the estimation of the mixing matrix $\bH$ from only two mixtures, achieved by CPD of three $2\times 2 \times \cdots \times 2$ derivative tensors of the second GCFs.
(b) Mean SAE (in dB) in estimation of the mixing matrix $\bH$ by CPD aided with a prior TT decomposition.}
\label{fig_bi_2xR_tt2cp}
\end{figure}
}

\begin{example}{\bf Blind separation of damped sinusoid signals.}\label{ex_bss_exp}
\end{example}

It is well-known that real or complex exponentials have a rank-1 structure. This is a perfect match for the scope of the proposed framework, as any linear combination of sinusoids can hence be decomposed in rank-1 terms.
The use of our proposed algorithm is next illustrated for the extraction of complex-valued damped sinusoids from a single mixture which is corrupted by Gaussian noise.
Consider a noisy signal, $y(t)$, created as a combination of $R = 3$ complex valued damped sinusoids, $x_r(t)$, to yield
\be
y(t) =  a_1 x_1(t)  +  a_2  x_2(t)  +  a_3  x_3(t)  + n(t) \notag ,
\ee
where
\be
x_r(t) = \exp(-i (\omega_r t  + \phi_r) - \tau_r t)  \notag ,
\ee
and $\omega_r = 20\pi r$, $\tau_r = 2r$, $\phi_r = \frac{\pi r}{2R+1}$, $t = 0, 1/300, \ldots, (T-1)/300$, and $T = 413$ samples.  The weights, $a_r$, were set such that
the component sources were equally contributing to the mixture.

In order to extract the source, $x_r(t)$, we adopted the method proposed in \cite{LathauwerTBSS} which comprises two steps: tensorization and tensor decomposition. More specifically, we first constructed from the signal $y(t)$ an order-4 Toeplitz tensor of size $192 \times 16 \times 16 \times 192$
\cite{cichocki2017tensor}, then reshaped it to an order-18 tensor of size $12 \times 2 \times 2 \times \cdots \times 2 \times 12$.
After such tensorization, each signal $x_r(t)$ yields a tensor $\tX_r$ of rank-1, while the observed signal $y(t)$ yields a tensor of rank-$R = 3$.
Hence, the approximation of this tensor by a CPD of rank-3 produced three rank-1 tensors, each being an estimation of the tensor $\tX_r$.

Algorithms~\ref{alg_TT_to_CPD_exact_svd} and \ref{alg_TT_to_CPD_als} were used to estimate the three rank-1 tensors.
The higher-order complex-valued tensors were first approximated by TT-tensors using the alternating double-core update (ADCU) algorithm\cite{Phan_TT_part1}.
The sources were then reconstructed from the estimated Toeplitz tensors. Mean and median values of SAEs (in dB) of the estimated signals over 100 independent runs are compared in Fig.~\ref{fig_bss_sinusoid}. The ordinary direct CPDs of high order tensors using the ALS algorithm could not retrieve the latent signals in most of the tests, although this algorithm succeeded in a few runs. Algorithm~\ref{alg_TT_to_CPD_exact_svd} for the exact TT-CPD conversion worked well, even for a high noise level. Its median SAEs were comparable with those achieved by CPD using Algorithm~\ref{alg_TT_to_CPD_als}, although its mean SAEs  were approximately 4-6 dB lower. The most stable results were achieved by CPD using Algorithm~\ref{alg_TT_to_CPD_als}.

{
\begin{figure}[t!]
\centering
{
\includegraphics[width=.85\linewidth, trim = 0.0cm 0cm 0cm 0cm,clip=true]{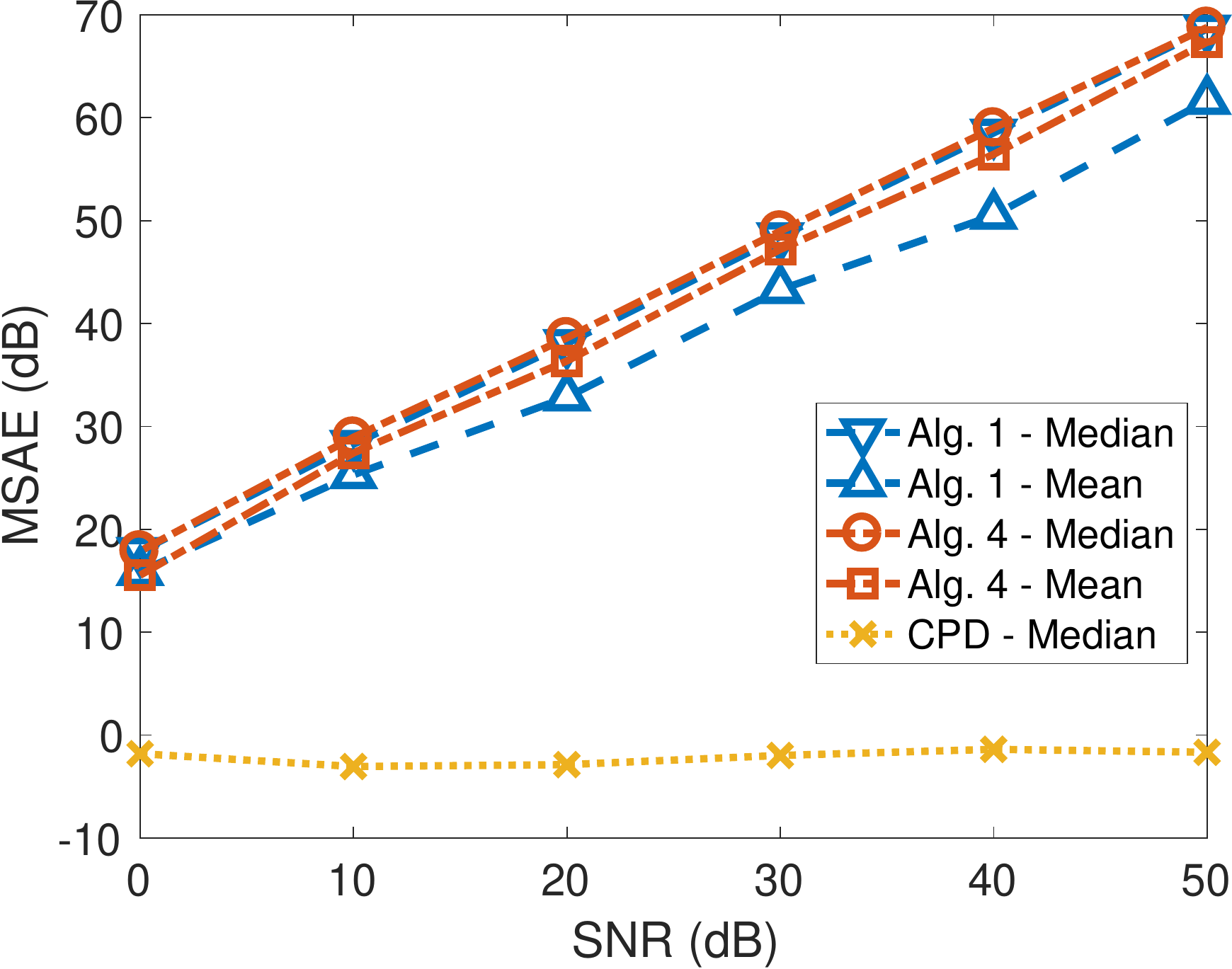}}
\caption{
Mean SAE (in dB) in the estimation of the complex-valued damped sinusoids from a single mixture through CPDs of order-18 tensors in Example~\ref{ex_bss_exp}.}
\label{fig_bss_sinusoid}
\end{figure}
}

Fig.~\ref{fig_bss_sinusoid_18vs20} confirms the efficiency of Algorithm~\ref{alg_TT_to_CPD} in another simulation {scenario}, where the signal $y(t)$ had a shorter length of $T = 123$. The Toeplitz tensors of order-6 and of size $48 \times 8\times 8\times 8 \times8 \times48$ were generated from the signal, {and were subsequently} reshaped to order-20 tensors of size $6 \times 2 \times 2 \times \cdots \times 2 \times 6$. Because of the shorter signal, the estimation accuracy was on average 4 dB of SAE worse than the results in the previous case. Nevertheless, we were still able to retrieve successfully the three complex-valued signals.

{
\begin{figure}[t!]
\centering
{
\includegraphics[width=.85\linewidth, trim = 0.0cm 0cm 0cm 0cm,clip=true]{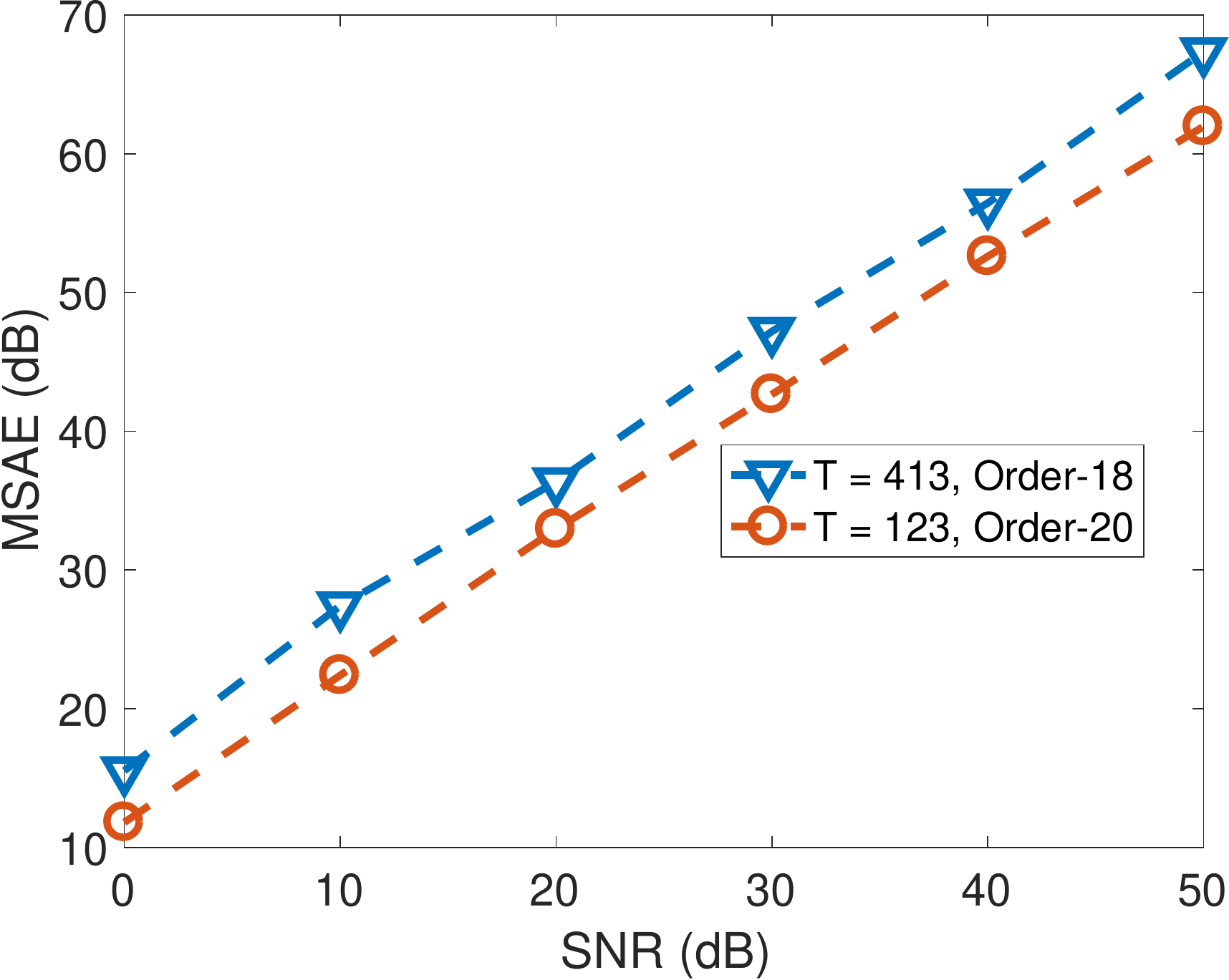}}
\caption{
Mean SAE (in dB) in the estimation of the complex-valued damped sinusoids from a single mixture of length $T = 123$ and $T = 413$ in Example~\ref{ex_bss_exp}.}
\label{fig_bss_sinusoid_18vs20}
\end{figure}
}

\begin{example}{\bf Low-rank approximation of a Hilbert tensor}
In this example, we approximated Hilbert tensors \cite{2016arXiv161102357S} of order $4$ and 6 and dimension $I_1= \cdots = I_N = 20$ defined as
\be
\tH(i_1, i_2, \ldots, i_N) = \frac{1}{i_1 + i_2 + \cdots +i_N - N+1} \notag \,.
\ee
For this tensor, we ran algorithms over 5000 iterations, but the decomposition could be stopped earlier if the consecutive approximation errors differed by less than $10^{-10}$.
The tensor was well approximated by a tensor of rank-$R = 7$, with a relatively error of $5\times 10^{-5}$, as shown in Fig.~\ref{fig_apx_Hilbert}.
The results show that this decomposition was quite challenging for the ALS algorithm. The non-linear least squares (NLS) algorithm \cite{Sorber-tensorlab} worked better than the ALS algorithm. For decomposition with the estimated ranks of $R = 6, 7,\ldots$, the NLS reached the maximum of iterations, and might need more iterations to achieve the best approximation errors.
The Levenberg-Marquardt (fLM) algorithm\cite{Phan_fLM} worked well in this example. We applied the TT2CP conversion and obtained compatible performances to those using fLM.

{
\begin{figure}[t!]
\centering
{
\includegraphics[width=.85\linewidth, trim = 0.0cm 0cm 0cm 0cm,clip=true]{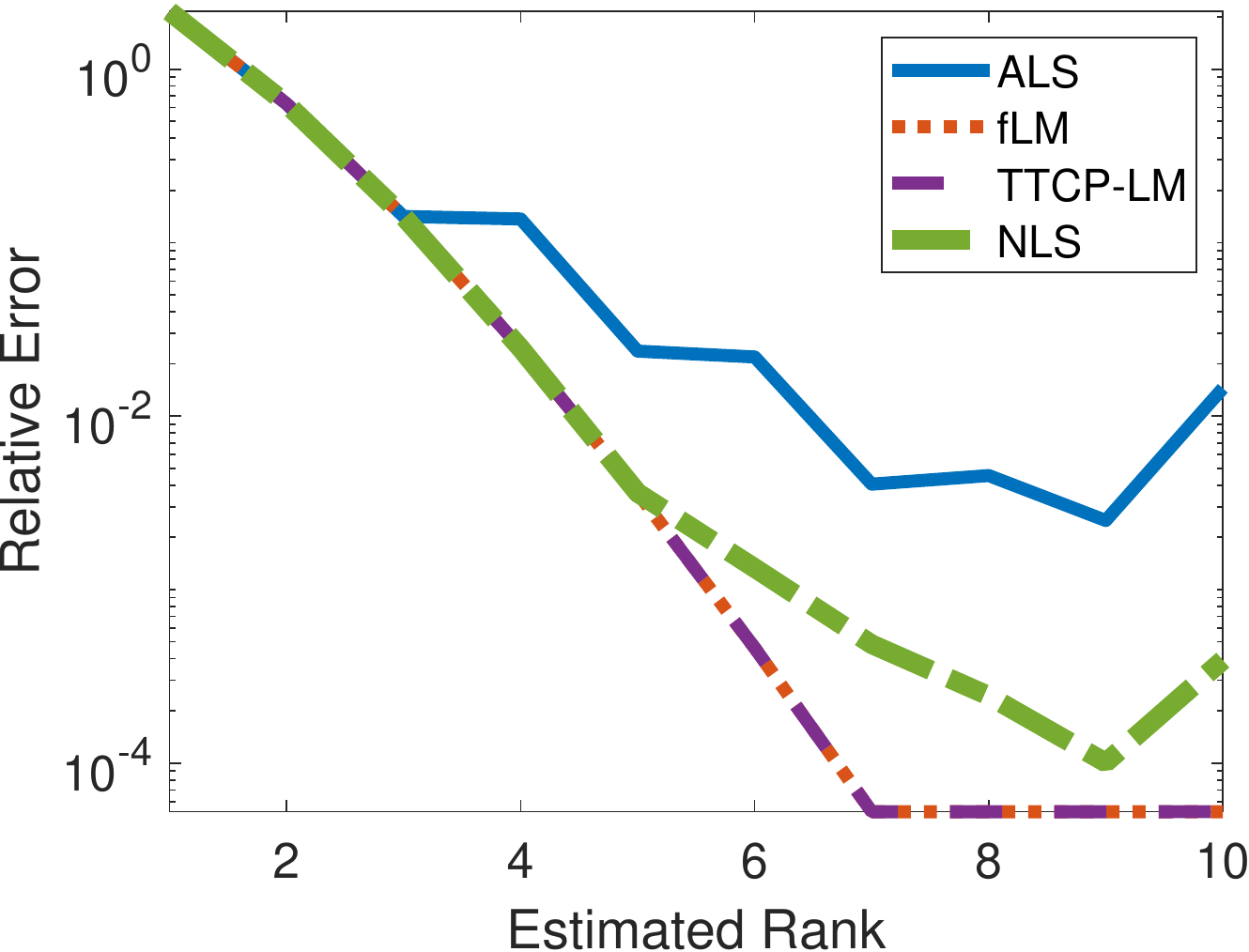}}
\includegraphics[width=.85\linewidth, trim = 0.0cm 0cm 0cm 0cm,clip=true]{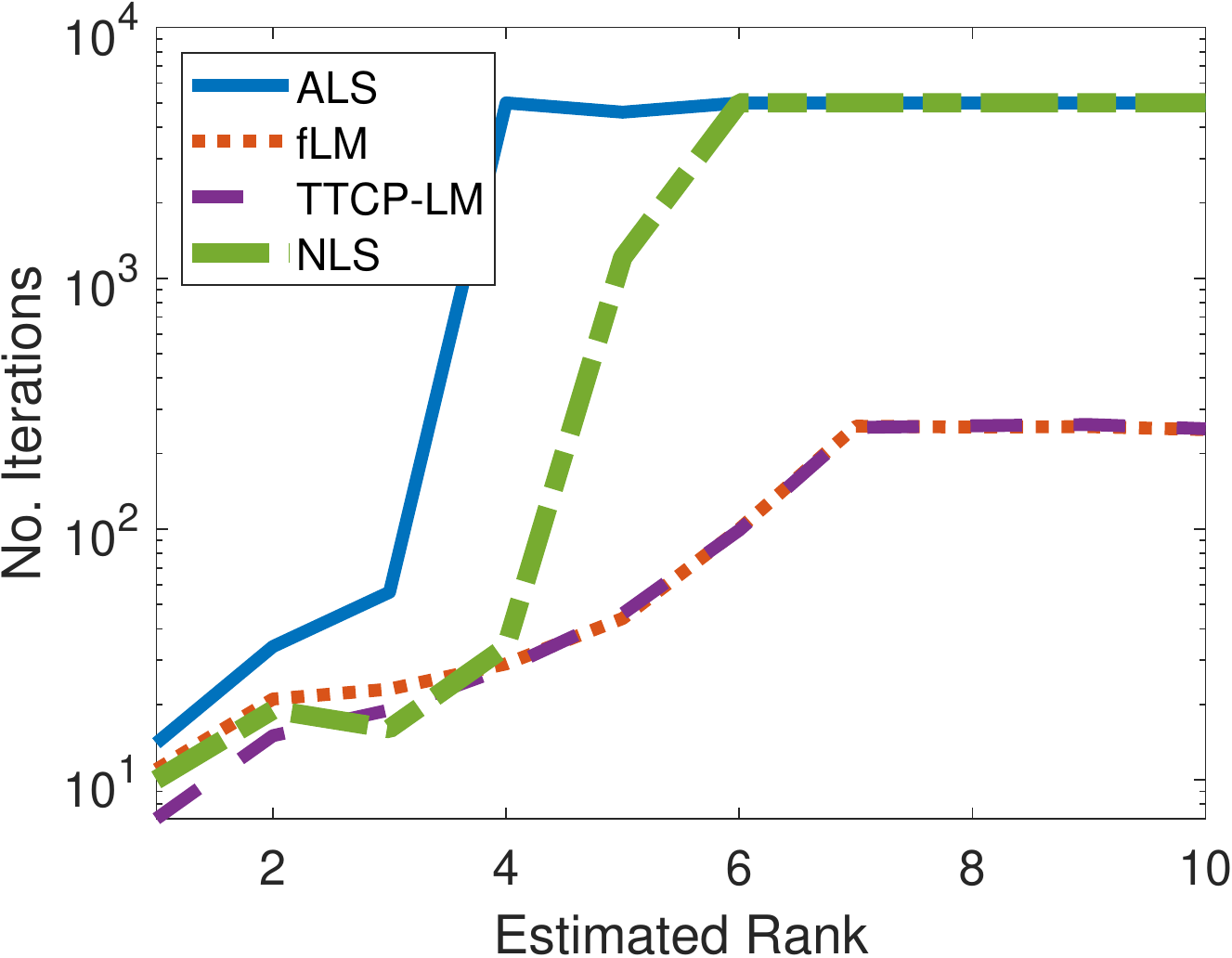}
\caption{Relative approximation errors and number of iterations of the CPD algorithms.}\label{fig_apx_Hilbert}
\end{figure}
}
\end{example}

{\bf Approximation with a predefined error bound.}
We next present low-rank approximations of the Hilbert tensor with exact error bounds of $\varepsilon = 10^{-2}$ and $10^{-3}$, that is
\be
\|\tY - \hat{\tY}\|_F = \varepsilon \, \|\tY\|_F
\ee
such that rank-1 tensor components of the estimated tensor had a minimum norm\cite{Phan_EPC}. This constrained decomposition is also known as the Error Preserving Correction (EPC) method.

We used the SQP algorithm for EPC, and initialized it by the leading singular vectors or tensors estimated using the CP-ALS algorithm.
The results were compared with those based on the TT2CPD method. More specifically, the Hilbert tensor was first approximated by a TT-tensor with an exact relative error $\varepsilon$ using the Alternating Single Core Update (ASCU)\cite{Phan_TT_part1}. The obtained TT-tensor had ranks of $(4,5,5)$ and was then approximated by Kruskal tensors with rank $R = 1, \ldots, 10$.

The relative approximation errors are compared in Fig.~\ref{fig_apx_Hilbert_epc}. For the relative approximation error bound of $10^{-2}$, the EPC obtained consistent results which matched the given error bound with a rank of $R=3$.
However, for the lower relative error bound of $10^{-3}$, EPC worked well and was stable only with the TT2CPD method. Notice that the tensor approximation should have rank of $R =5$ to attain the required error bound.

For decomposition of the Hilber tensor of order-6, EPC obtained the relative error bound of $\varepsilon = 10^-3$ using diffenent initialization methods. However, for approximations with a lower relative error of $10^{-4}$,  with only TT2CPD, the EPC method achieved the desired  goal. The relative errors are compared in Fig.~\ref{fig_apx_Hilbert6_epc}.

\begin{figure}[t]
\centering
\includegraphics[width=.85\linewidth, trim = 0.0cm 0cm 0cm 0cm,clip=true]{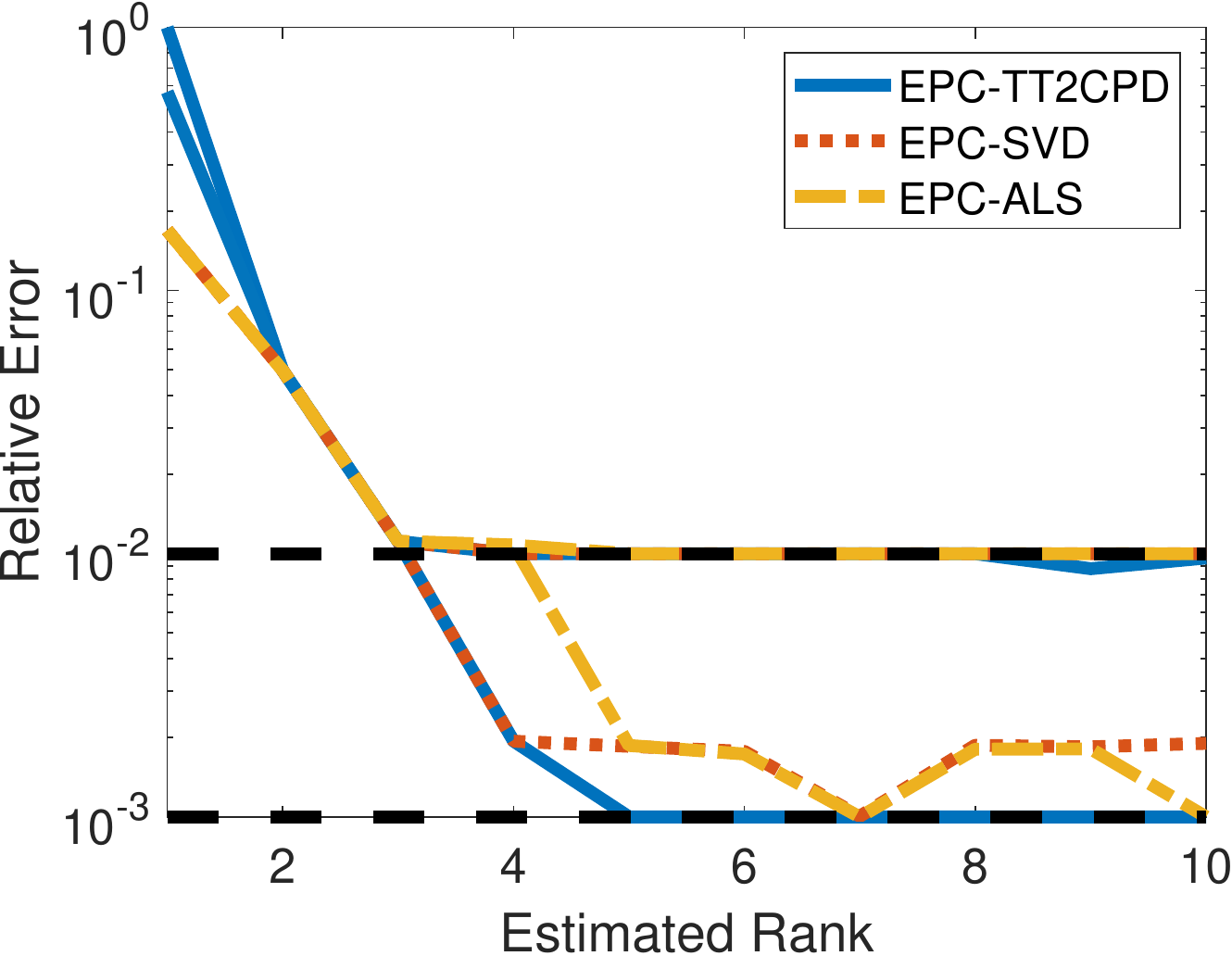}
\caption{Relative errors of the EPC using different initialization methods for decomposition of the Hilbert tensor of order 4.}\label{fig_apx_Hilbert_epc}
\end{figure}

\begin{figure}[t]
\centering
\includegraphics[width=.85\linewidth, trim = 0.0cm 0cm 0cm 0cm,clip=true]{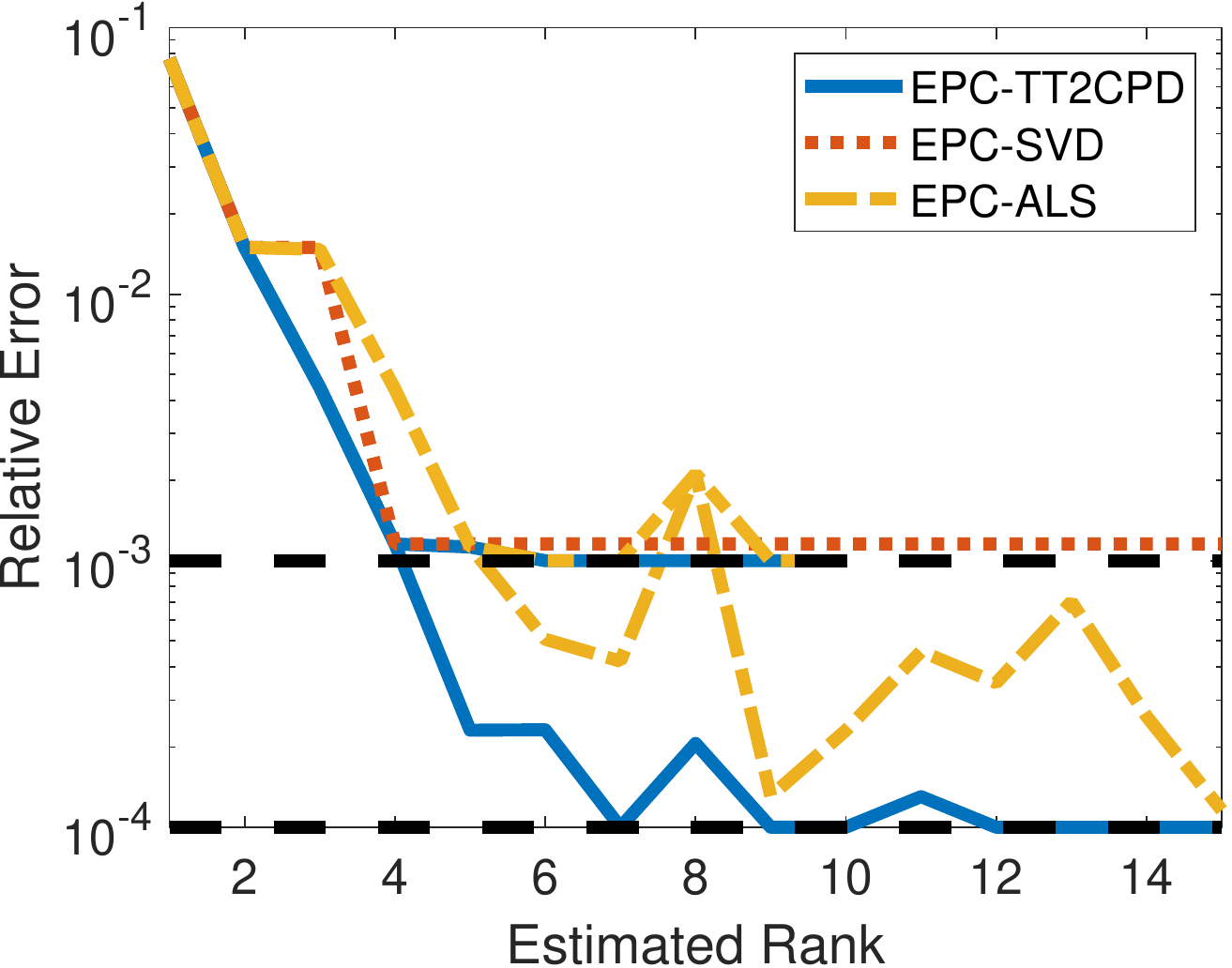}
\caption{Relative errors of the EPC using different initialization methods for decomposition of the Hilbert tensor of order 6.}\label{fig_apx_Hilbert6_epc}
\end{figure}

\section{Conclusions and Extensions}\label{sec:conclusion}

We have presented a novel application of the Tensor Train (TT) decomposition, a type of tensor networks to the calculation of Canonical Polyadic Decomposition (CPD) of higher-order tensors.
The proposed method has been shown to provide a general framework and include an exact conversion from TT-decomposition to CPD and an iterative algorithm to estimate CPD from a TT representation.
The proposed method can also be used to provide good initials for the constrained CPD.
Finally, a similar conversion can be derived from the tensor chain, a closed tensor network \cite{Espig2012,2018arXiv180710247L} to the CP shallow network. 
Simulation studies have verified the abilities of the proposed approach to both accurate estimation of the tensor rank and efficient computation of CPD of higher-order tensors, both long standing critical issues in tensor manipulation.

\appendices

\section{Computation of the Gradients for the Structured Kruskal Tensors}\label{sec:fastCP_structured_tensor}

Let $\bQ_n = [\tG_n]_{(2)}^T \bA^{(n)}$ be matrices of size $R_nR_{n+1} \times R$. From (\ref{eq_Un_tt_kt}), we have
\be
\bU^{(n) T}\bA^{(n)} &=  \left( \1_{R_{>(n+1)}} \otimes \bI_{R_nR_{n+1}} \otimes \1_{R_{<n}}  \right) \bQ_n  
=  \1_{R_{>(n+1)}} \otimes \bQ_n  \otimes \1_{R_{<n}} \,. \notag 
\ee
We next define matrices $\bPhi_{<n}$ and $\bPhi_{>n}$ of
sizes $R_1R_2\cdots R_n \times R$ and $R_{n+1}R_{n+2}\cdots R_N \times R$, respectively, as 
\be
\bPhi_{<n} &=& \bigCircleDast_{k =1}^{n-1} \left( \1_{R_{k+2:n}} \otimes \bQ_k  \otimes \1_{R_{<k}}\right) \notag \\
&=& \left(\bQ_{n-1} \otimes \1_{R_{<(n-1)}}\right) \*  \left(\1_{R_{n}} \otimes \bPhi_{<(n-1)}\right) \label{eq_psi_leftn} ,\\
\bPhi_{>n} &=& \bigCircleDast_{k = n+1}^{N} \left( \1_{R_{>k}} \otimes \bQ_k  \otimes \1_{R_{n+1:k-1}}\right) \notag \\
&=& \left(\1_{R_{>n+2}} \otimes \bQ_{n+1}\right) \*  \left( \bPhi_{>(n+1)} \otimes \1_{R_{n+1}} \right)\label{eq_psi_rightn}\,.
\ee
It then follows that
$\bigCircleDast_{k < n} \bU^{(k) T}\bA^{(k)} = \1_{R_{>n}} \otimes \bPhi_{<n}$ and $\bigCircleDast_{k > n} \bU^{(k) T}\bA^{(k)} = \bPhi_{>n}\otimes \1_{R_{<n+1}}$, 
and the Hadamard product of all-but-one matrices $\bU^{(k) T}\bA^{(k)}$ is equivalent to the Khatri-Rao product of the two matrices $\bPhi_{<n}$ and $\bPhi_{>n}$
\be
\left(\bigCircleDast_{k\neq n} \bU^{(k) T}\bA^{(k)} \right) &=& (\1_{R_{>n}} \otimes \bPhi_{<n}) \* (\bPhi_{>n} \otimes \1_{R_{<n+1}}) \notag \\
&=& \bPhi_{>n} \odot \bPhi_{<n} \,.
\ee
Now, we can rewrite the gradient in (\ref{eq_gradient_An}) by taking into account the linear dependence structure of $\bU^{(n)}$ in (\ref{eq_Un_tt_kt}) and
$\bI_{R_n R_{n+1}} = \bI_{R_{n+1}} \otimes \bI_{R_n}$, to give
\begin{align}
\frac{\partial D}{\partial \bA^{(n)}} &= \bU^{(n)} \left(\bigCircleDast_{k\neq n} \bU^{(k) T}\bA^{(k)} \right) \notag \\
&=
[\tG_{n}]_{(2)} \, (\1_{R_{>(n+1)}}^T \otimes \bI_{R_n R_{n+1}} \otimes \1_{R_{<n}}^T) (\bPhi_{>n} \odot \bPhi_{<n})\notag \\
&= [\tG_{n}]_{(2)} \, \left((\1_{R_{>(n+1)}}^T \otimes \bI_{R_{n+1}})  \bPhi_{>n}  \odot (\bI_{R_{n}} \otimes \1_{R_{<n}}^T)\bPhi_{<n}\right)\notag\,\\
&= [\tG_{n}]_{(2)} \, \left(\bPsi_{>n}  \odot \bPsi_{<n}\right)\, \label{eq_gradient_Un_2}
\end{align}
where the two matrices $\bPsi_{<n} = (\bI_{R_{n}} \otimes \1_{R_{<n}}^T)\bPhi_{<n}$ and $\bPsi_{>n} =(\1_{R_{>(n+1)}}^T \otimes \bI_{R_{n+1}})  \bPhi_{>n}$ are of size $R_n \times R$ and $R_{n+1} \times R$, respectively.

From (\ref{eq_psi_leftn}), each column $\bpsi^{(<n)}_r$, $r = 1, 2, \ldots, R$, of the matrix $\bPsi_{<n}$ can be represented as
\begin{align}
\bpsi^{(<n)}_r &=  (\bI_{R_{n}} \otimes \1_{R_{<n}}^T) \left( \left(\bq^{(n-1)}_{r} \otimes \1_{R_{<(n-1)}}\right) \*  \left(\1_{R_{n}} \otimes \bphi^{(<(n-1))}_{r}\right)\right) \notag \\
&=  \left((\bI_{R_{n}} \otimes \1_{R_{<n}}^T)\diag(\bq^{(n-1)}_{r}) \otimes \1_{R_{<(n-1)}}^T \right)\left(\1_{R_{n}} \otimes \bphi^{(<(n-1))}_{r}\right) \notag \\
&= (\bI_{R_{n}} \otimes \1_{R_{<n}}^T)\diag(\bq^{(n-1)}_{r}) \left(\bI_{R_{n-1}R_{n}} \otimes \1_{R_{<(n-1)}}^T \right)\left(\1_{R_{n}} \otimes \bphi^{(<(n-1))}_{r}\right) \notag \\
&= (\bI_{R_{n}} \otimes \1_{R_{<n}}^T)\diag(\bq^{(n-1)}_{r}) \left(\1_{R_n}\otimes \left(\bI_{R_{n-1}} \otimes \1_{R_{<(n-1)}}^T \right) \bphi^{(<(n-1))}_{r}\right) \notag \\
&= (\bI_{R_{n}} \otimes \1_{R_{<n}}^T)\diag(\bq^{(n-1)}_{r}) \left(\1_{R_n}\otimes \bphi^{(<(n-1))}_{r}\right) \notag\\
&= \bQ^{(n-1)\, T}_r \,   \bpsi^{(<(n-1))}_{r} \label{eq_psi_left_r1} \\
&= \tG_n \bar{\times}_1 \, \bpsi^{(<(n-1))}_{r}
\, \bar{\times}_2  \ba^{(n)}_r  \,  \label{eq_psi_left_r2}
\end{align}
where $\bQ^{(n-1)}_r$ is a matrix of size $R_{n-1} \times R_n$ and its vectorization is the $r$-th column of the matrix $\bQ_n$, i.e., $\vtr{\bQ^{(n-1)}_r} = \bq^{(n-1)}_{r}$.
The last two expressions form a recursive formula which can efficiently compute the matrices $\bPsi_{<n}$, and can be rewritten as
\be
    \bPsi_{<n} = [\tG_n]_{(3)} (\bA^{(n)} \odot \bPsi_{<(n-1)}) \, .\label{eq_psi_left_n_1}
\ee
Similarly, we can derive a recursive expression for the columns $\bpsi^{(>n)}_r$ of $\bPsi_{>n}$ as
\begin{align}
\bpsi^{(>n)}_r = (\1_{R_{>(n+1)}}^T \otimes \bI_{R_{n+1}}) \bphi^{(>n)}_r 
= \bQ^{(n+1)}_r \bpsi^{(>(n+1))}_r \notag
\end{align}
and
\be
\bPsi_{>n} = \left[\tG_{n+1}   \right]_{(1)}    \left( \bPsi_{>(n+1)} \odot  \bA^{(n+1)} \right)\,.\label{eq_Psi_rightn}
\ee
We note that the matrices $\bPhi_{<n}$ and $\bPhi_{>n}$ vanish in the expressions for the computation of $\bPsi_{<n}$ in (\ref{eq_psi_left_n_1})
and $\bPsi_{>n}$ in (\ref{eq_Psi_rightn}). The gradients can be finally computed from $\bPsi_{<n}$ and  $\bPsi_{>n}$ as given in (\ref{eq_gradient_Un_2}).
%

 \section*{Acknowledgment}
 The first author wishes to thank L. De Lathauwer for the helpful discussion and his suggestion of the exact sequential conversion method. 
The work of A.H.P, A.C, I.O and S.A.A was supported by the Mega Grant project (14.756.31.0001).
\bibliographystyle{IEEEbib}
\bibliography{BIBTENSORS2018}

\begin{thebibliography}{10}

\bibitem{Hitchcock1927}
F.L. Hitchcock,
\newblock ``Multiple invariants and generalized rank of a $p$-way matrix or
  tensor,''
\newblock {\em Journal of Mathematics and Physics}, vol. 7, pp. 39--79, 1927.

\bibitem{Harshman72}
R.~A. Harshman,
\newblock ``Determination and proof of minimum uniqueness conditions for
  {PARAFAC1},''
\newblock {\em UCLA Working Papers in Phonetics}, vol. 22, 1972.

\bibitem{Carroll_Chang}
J.D. Carroll and J.J. Chang,
\newblock ``Analysis of individual differences in multidimensional scaling via
  an $n$-way generalization of {E}ckart--{Y}oung decomposition,''
\newblock {\em Psychometrika}, vol. 35, no. 3, pp. 283--319, 1970.

\bibitem{kruskal77}
J.B. Kruskal,
\newblock ``Three-way arrays: {R}ank and uniqueness of trilinear
  decompositions, with application to arithmetic complexity and statistics,''
\newblock {\em Linear Algebra Appl.}, vol. 18, pp. 95--138, 1977.

\bibitem{Sidiropoulos00Bro}
N.~Sidiropoulos, R.~Bro, and G.~Giannakis,
\newblock ``Parallel factor analysis in sensor array processing,''
\newblock {\em IEEE Transactions on Signal Processing}, vol. 48, no. 8, pp.
  2377--2388, 2000.

\bibitem{YeredorCaf}
A.~Yeredor,
\newblock ``Blind source separation via the second characteristic function,''
\newblock {\em Signal Processing}, vol. 80, no. 5, pp. 897--902, 2000.

\bibitem{Comon20062271}
P.~Comon and M.~Rajih,
\newblock ``Blind identification of under-determined mixtures based on the
  characteristic function,''
\newblock {\em Signal Processing}, vol. 86, no. 9, pp. 2271 -- 2281, 2006,
\newblock Special Section: Signal Processing in \{UWB\} Communications.

\bibitem{sorensenVDM}
M.~S{\o}rensen and L.~{De Lathauwer},
\newblock ``Blind signal separation via tensor decomposition with {V}andermonde
  factor. {P}art {I}: {C}anonical polyadic decomposition,''
\newblock {\em IEEE Transactions on Signal Processing}, vol. 61, no. 22, pp.
  5507--5519, 2013.

\bibitem{8421043}
Y.~Wu, H.~Tan, Y.~Li, J.~Zhang, and X.~Chen,
\newblock ``A fused cp factorization method for incomplete tensors,''
\newblock {\em IEEE Transactions on Neural Networks and Learning Systems}, pp.
  1--14, 2018.

\bibitem{8305626}
X.~Chen, Z.~Han, Y.~Wang, Q.~Zhao, D.~Meng, L.~Lin, and Y.~Tang,
\newblock ``A generalized model for robust tensor factorization with noise
  modeling by mixture of gaussians,''
\newblock {\em IEEE Transactions on Neural Networks and Learning Systems}, pp.
  1--14, 2018.

\bibitem{BayesCPD-TNNLS}
Q.~Zhao, G.~Zhou, L.~Zhang, A.~Cichocki, and S.~Amari,
\newblock ``Bayesian robust tensor factorization for incomplete multiway
  data,''
\newblock {\em IEEE Transactions on Neural Networks and Learning Systems}, vol.
  PP, no. 99, pp. 1--1, 2016.

\bibitem{8116755}
F.~Ju, Y.~Sun, J.~Gao, Y.~Hu, and B.~Yin,
\newblock ``Vectorial dimension reduction for tensors based on bayesian
  inference,''
\newblock {\em IEEE Transactions on Neural Networks and Learning Systems}, pp.
  1--14, 2018.

\bibitem{sid_adPARAFAC}
D.~Nion and N.D. Sidiropoulos,
\newblock ``{Adaptive Algorithms to Track the PARAFAC Decomposition of a
  Third-Order Tensor},''
\newblock {\em IEEE Transactions on Signal Processing}, vol. 57, no. 6, pp.
  2299--2310, June 2009.

\bibitem{Phan_tensordeflation_alg}
A.-H. Phan, P.~Tichavsk{\'y}, and A.~Cichocki,
\newblock ``Tensor deflation for {CANDECOMP/PARAFAC}. {Part 1}: {Alternating
  Subspace Update Algorithm},''
\newblock {\em IEEE Transaction on Signal Processing}, vol. 63, no. 12, pp.
  5924--5938, 2015.

\bibitem{DBLP:journals/corr/JaderbergVZ14}
M.~Jaderberg, A.~Vedaldi, and A.~Zisserman,
\newblock ``Speeding up convolutional neural networks with low rank
  expansions,''
\newblock {\em CoRR}, vol. abs/1405.3866, 2014.

\bibitem{DBLP:journals/corr/LebedevGROL14}
V.~Lebedev, Y.~Ganin, M.~Rakhuba, Ivan~V. Oseledets, and V.~S. Lempitsky,
\newblock ``Speeding-up convolutional neural networks using fine-tuned
  cp-decomposition,''
\newblock {\em CoRR}, vol. abs/1412.6553, 2014.

\bibitem{Strassen:1969:GEO:2722431.2722798}
V.~Strassen,
\newblock ``Gaussian elimination is not optimal,''
\newblock {\em Numer. Math.}, vol. 13, no. 4, pp. 354--356, Aug. 1969.

\bibitem{TICHAVSKY2017362}
Petr Tichavsk{\'y}, Anh-Huy Phan, and Andrzej Cichocki,
\newblock ``Numerical {CP} decomposition of some difficult tensors,''
\newblock {\em Journal of Computational and Applied Mathematics}, vol. 317, pp.
  362 -- 370, 2017.

\bibitem{Phan_FastALS}
A.-H. Phan, P.~Tichavsk{\'y}, and A.~Cichocki,
\newblock ``Fast alternating {LS} algorithms for high order {CANDECOMP/PARAFAC}
  tensor factorizations,''
\newblock {\em Signal Processing, IEEE Transactions on}, vol. 61, no. 19, pp.
  4834--4846, 2013.

\bibitem{doi:10.1137/14097968X}
N.~Vannieuwenhoven, K.~Meerbergen, and R.~Vandebril,
\newblock ``Computing the gradient in optimization algorithms for the cp
  decomposition in constant memory through tensor blocking,''
\newblock {\em SIAM Journal on Scientific Computing}, vol. 37, no. 3, pp.
  C415--C438, 2015.

\bibitem{Comon-ALS09}
P.~Comon, X.~Luciani, and A.~L.~F. de~Almeida,
\newblock ``Tensor decompositions, alternating least squares and other tales,''
\newblock {\em Journal of Chemometrics}, vol. 23, pp. 393--405, 2009.

\bibitem{Phan_FCP}
A.-H. Phan, P.~Tichavsk{\'y}, and A.~Cichocki,
\newblock ``{CANDECOMP/PARAFAC} decomposition of high-order tensors through
  tensor reshaping,''
\newblock {\em IEEE Transactions on Signal Processing}, vol. 61, no. 19, pp.
  4847--4860, 2013.

\bibitem{Bhaskara:2014:SAT:2591796.2591881}
A.~Bhaskara, M.~Charikar, A.~Moitra, and A.~Vijayaraghavan,
\newblock ``Smoothed analysis of tensor decompositions,''
\newblock in {\em Proceedings of the Forty-sixth Annual ACM Symposium on Theory
  of Computing}, New York, NY, USA, 2014, STOC '14, pp. 594--603, ACM.

\bibitem{doi:10.1137/16M1090132}
L.~Chiantini, G.~Ottaviani, and N.~Vannieuwenhoven,
\newblock ``Effective criteria for specific identifiability of tensors and
  forms,''
\newblock {\em SIAM Journal on Matrix Analysis and Applications}, vol. 38, no.
  2, pp. 656--681, 2017.

\bibitem{Klumper91}
A.~Klumper, A.~Schadschneider, and J.~Zittartz,
\newblock ``Equivalence and solution of anisotropic spin-1 models and
  generalized t-j fermion models in one dimension,''
\newblock {\em Journal of Physics A: Mathematical and General}, vol. 24, no.
  16, pp. L955, 1991.

\bibitem{Vidal03}
G.~Vidal,
\newblock ``Efficient classical simulation of slightly entangled quantum
  computations,''
\newblock {\em Physical Review Letters}, vol. 91, no. 14, pp. 147902, 2003.

\bibitem{OseledetsTT09}
I.V. Oseledets and E.E. Tyrtyshnikov,
\newblock ``Breaking the curse of dimensionality, or how to use {SVD} in many
  dimensions,''
\newblock {\em SIAM Journal on Scientific Computing}, vol. 31, no. 5, pp.
  3744--3759, 2009.

\bibitem{cichocki2016tensor}
A.~Cichocki, N.~Lee, I.~Oseledets, A.-H. Phan, Q.~Zhao, and D.~P Mandic,
\newblock ``Tensor networks for dimensionality reduction and large-scale
  optimization: Part 1 low-rank tensor decompositions,''
\newblock {\em Foundations and Trends{\textregistered} in Machine Learning},
  vol. 9, no. 4-5, pp. 249--429, 2016.

\bibitem{Phan_TT_part1}
A.-H. {Phan}, A.~{Cichocki}, A.~{Uschmajew}, P.~{Tichavsky}, G.~{Luta}, and
  D.~{Mandic},
\newblock ``Tensor networks for latent variable analysis. {Part I: Algorithms}
  for tensor train decomposition,''
\newblock {\em ArXiv e-prints}, 2016.

\bibitem{OseledetsTT11}
I.V. Oseledets,
\newblock ``Tensor-train decomposition,''
\newblock {\em SIAM Journal on Scientific Computing}, vol. 33, no. 5, pp.
  2295--2317, 2011.

\bibitem{CEM:CEM1206}
R.~Bro, R.~A. Harshman, N.~D. Sidiropoulos, and M.~E. Lundy,
\newblock ``Modeling multi-way data with linearly dependent loadings,''
\newblock {\em Journal of Chemometrics}, vol. 23, no. 7-8, pp. 324--340, 2009.

\bibitem{Sanchez1990}
E.~Sanchez and B.R. Kowalski,
\newblock ``Tensorial resolution: a direct trilinear decomposition,''
\newblock {\em Journal of Chemometrics}, vol. 4, pp. 29--45, 1990.

\bibitem{Bader_kolda}
B.W. Bader and T.G. Kolda,
\newblock ``Algorithm 862: {MATLAB} tensor classes for fast algorithm
  prototyping,''
\newblock {\em ACM Transactions on Mathematical Software}, vol. 32, no. 4, pp.
  635--653, 2006.

\bibitem{tensorbox}
A.H. Phan, P.~Tichavsk{\'y}, and A.~Cichocki,
\newblock ``{TENSORBOX: MATLAB package for tensor decomposition},'' 2012.

\bibitem{Phan_fLM}
A.-H. Phan, P.~Tichavsk{\'y}, and A.~Cichocki,
\newblock ``Low complexity damped {Gauss-Newton} algorithms for
  {CANDECOMP/PARAFAC},''
\newblock {\em SIAM Journal on Matrix Analysis and Applications}, vol. 34, no.
  1, pp. 126--147, 2013.

\bibitem{Phan_EPC}
A.-H. Phan, P.~Tichavsk\'{y}, and A.~Cichocki,
\newblock ``Error preserving correction: A method for {CP} decomposition at a
  target error bound,''
\newblock {\em arXiv preprint}, 2018.

\bibitem{Petr_CRIB}
P.~Tichavsk{\'y}, A.-H. Phan, and Z.~Koldovsk{\'y},
\newblock ``{Cram{\'e}r-Rao}-induced bounds for {CANDECOMP/PARAFAC} tensor
  decomposition,''
\newblock {\em IEEE Transactions on Signal Processing}, vol. 61, no. 8, pp.
  1986--1997, 2013.

\bibitem{cichocki2017tensor}
A.~Cichocki, A.-H. Phan, Q.~Zhao, M.~Lee, I.~Oseledets, M.~Sugiyama, and D.~P
  Mandic,
\newblock ``Tensor networks for dimensionality reduction and large-scale
  optimization: Part 2 applications and future perspectives,''
\newblock {\em Foundations and Trends{\textregistered} in Machine Learning},
  vol. 9, no. 6, pp. 431--673, 2017.

\bibitem{LathauwerTBSS}
L.~{D}e {L}athauwer,
\newblock ``Blind separation of exponential polynomials and the decomposition
  of a tensor in rank- $({L}_r,{L}_r,1)$ terms,''
\newblock {\em SIAM Journal on Matrix Analysis and Applications}, vol. 32, no.
  4, pp. 1451--1474, 2011.

\bibitem{2016arXiv161102357S}
Y.~{Song} and L.~{Qi},
\newblock ``{Infinite dimensional Hilbert tensors on spaces of analytic
  functions},''
\newblock {\em ArXiv e-prints}, Nov. 2016.

\bibitem{Sorber-tensorlab}
L.~Sorber, M.~{Van Barel}, and L.~{De Lathauwer},
\newblock ``{Tensorlab v1.0},'' Feb. 2013.

\bibitem{Espig2012}
M.~Espig, K.~K. Naraparaju, and J.~Schneider,
\newblock ``A note on tensor chain approximation,''
\newblock {\em Computing and Visualization in Science}, vol. 15, no. 6, pp.
  331--344, Dec 2012.

\bibitem{2018arXiv180710247L}
Y.~{Ling}, Y.~{Liu}, Z.-Y. {Xian}, and Y.~{Xiao},
\newblock ``{Tensor chain and constraints in tensor networks},''
\newblock {\em ArXiv e-prints}, July 2018.

\end{thebibliography}

\end{document}